\documentclass[final,onefignum,onetabnum]{siamart190516}

\usepackage[percent]{overpic}
\usepackage{graphicx,graphics,epsfig}
\usepackage{amsmath,amssymb,dsfont}
\usepackage{mathrsfs}
\usepackage{arydshln}
\usepackage{url,cite}
\usepackage{enumerate}
\usepackage{color}

\usepackage[normalem]{ulem}
\usepackage{soul}

\usepackage{setspace}
\usepackage{multirow,booktabs}
\usepackage{caption}
\usepackage{subfigure}
\def\R{\mathbb{R}}
\def\H{\mathcal{H}}

\def\eref#1{{\rm (\ref{#1})}}
\def\fref#1{{Fig.~\ref{#1}}}
\def\sref#1{{Sect.~\ref{#1}}}
\def\tref#1{{Tab.~\ref{#1}}}

\def\lref#1{{Lem.~\ref{#1}}}

\def\qed{~\relax\ifmmode\hskip2em \Box
 \else\unskip\nobreak\hskip1em \hfill$\Box$
 \fi \newline}

\def\dt#1{{\dot{#1}}}
\def\dtt#1{{\ddot{#1}}}
\def\icz{u_0}
\def\icf{u_1}

\def\rbfHm{\H^\rbfm}
\def\solHm{\H^\solm}
\def\rbfm{{m}}
\def\solm{{m}}
\def\Phim{\Phi_{\rbfm}}
\def\trialsp{{\mathcal{V}_{X,\Sigma,\Phim}}}
\def\interp{{\mathcal{I}_{X,\Sigma,\Phim}}}
\def\pde{\eref{eq:NLW}--\eref{eq:AssumOfu} }
\def\semipde{\eref{eq:NLWGalerkinDiscrete2}--\eref{eq:NLWGalerkinDiscrete2IC}}
\def\V{{\mathcal{U}^1}(\Omega)}
\def\Vm{{\mathcal{U}^\solm}(\Omega)}
\def\Po{{\mathcal{P}}}
\def\RN{{\mathcal{R}}}
\def\uoft{{u(t)}}
\def\IP#1#2{{\big\langle{#1},\,{#2}\big\rangle}}
\def\IPOmega#1#2{{{\IP{#1}{#2}}_{L^2(\Omega)}}}
\def\balpha{{\boldsymbol{\alpha}}}
\def\bbeta{{\boldsymbol{\beta}}}
\def\bgamma{{\boldsymbol{\gamma}}}
\def\bvartheta{{\boldsymbol{\vartheta}}}
\def\mtxA{{G_{\Omega,\Phim}}}
\def\mtxB{{G_{\Omega,\nabla\Phim}}}

\graphicspath{{figures/}}

\title{A kernel-based meshless conservative Galerkin method for solving Hamiltonian wave equations\thanks{
    This work is supported by National Natural Sicence Foundation of China (12101310, 10631015), Natural Science Foundation of Jiangsu Province (BK20210315), 2021 Jiangsu Shuangchuang Talent Program (JSSCBS 20210222) and a Hong Kong Research Grant Council GRF Grant.
}}

\author{
    Zhengjie Sun\thanks{School of Mathematics and Statistics,  Nanjing University of Science and Technology, Nanjing, China (\email{zhengjiesun@njust.edu.cn}).}
    \and
    Leevan Ling\thanks{Department of Mathematics, Hong Kong Baptist University, Kowloon Tong, Hong Kong (\email{lling@hkbu.edu.hk}).}
}

\begin{document}

\maketitle
\begin{abstract}
We propose a meshless conservative Galerkin method for 
solving Hamiltonian wave equations. We first discretize the equation in space using radial basis functions in a Galerkin-type formulation.
Differ from the traditional RBF Galerkin method that directly uses nonlinear functions in its weak form, our method employs  appropriate projection operators in the construction of the Galerkin equation, which will be shown to conserve global energies.
Moreover, we provide a complete error analysis to the proposed discretization.
We further derive the fully discretized solution by a second order
average vector field  scheme.
We prove that the fully discretized solution preserved the discretized energy exactly. Finally, we provide some numerical examples to demonstrate the accuracy and the energy conservation.
\end{abstract}
\begin{keywords}
Energy conservation; meshless Galerkin method; radial basis functions; semilinear wave equations; evolutionary PDEs; average vector field.
\end{keywords}
\begin{AMS}
    65D05, 65M12, 65M60, 65N40.
\end{AMS}

\section{Introduction}
\label{Sec:Introduction}
We consider the following semilinear Hamiltonian  wave equation,
which is a classical example of Hamiltionian PDEs that have a symplectic geometric structure, in the form of
\begin{equation}\label{eq:NLW}
\dtt{u}(x,t)-\Delta u(x,t)+F'\big(u(x,t)\big)=0, \qquad (x,t)\in \R^d\times(0,T]
\end{equation}
for some smooth real-valued function $F':\mathbb{R}\rightarrow\mathbb{R}$, subject to the initial conditions
\begin{equation}\label{eq:NLWInit}
    u(x,0)=\icz(x) \mbox{\quad and \quad} \dt{u}(x,0)=\icf(x), \qquad x\in\R^d.
\end{equation}
Here, the conventional notation $F'(u)$ for nonlinearity is used as in the literature \cite{Reich,Zhu0,WuZhang}. We focus on bounded and compactly supported waves; that is, the support of the solution $u$ for all time, i.e.,
\begin{equation}\label{eq:AssumOfu}
    \Sigma:=\text{supp}(u)=\big\{x\in\R^d\,\big|\,u(x,t)\neq 0 \mbox{ for some } t\in[0,T] \big\},
\end{equation}
is a compact set. It is assumed that the supports of both $\icz$ and $\icf$ were also contained in $\Sigma$.
Equations \pde describe vibration and wave propagation phenomenon. They appear frequently in symplectic geometry \cite{hairer2006geometric,feng2010symplectic}  as well as many areas of physics, including quantum mechanics and superconductors.  Different choice of $F'(u)$ yields mathematical model for different physics. Besides of the trivial linear case  $F'(u)=Cu$ that models wave propagation, our study also covers the Sine-Gordon equation with $F'(u)=\sin u$,
the Klein-Gordon equation with $F'(u)=|u|^p u$ for some $p\geq0$,
the exponential wave equation with $F'(u)=C\exp(-u)$,
and many more.

Under the assumption in \eref{eq:AssumOfu}, there exists some sufficiently large (to be specified in \sref{Sec:Radial basis functions and approximation spaces}), open, connected, bounded domain $\Omega \supset \supset \Sigma$ with smooth boundary, on which the Hamiltonian equation is subject to zero Dirichlet boundary conditions.
Existence and uniqueness of Hamiltonian equation \pde were studied extensively in the literature. For example, in the Klein-Gordon equation, Lions \cite{lions1969quelques} proved that \pde posed in bounded domains has a unique solution, provided that the initial data satisfies $\icz(x)\in \H_0^1(\Omega)\cap L^{p+2}(\Omega)$ and $\icf(x)\in L^2(\Omega)$.
Later, \cite{evans2010partial,gibibre1985global} proved that smooth solution exists when $p<4$. More recently, it was proven theoretically and numerically in \cite{medeiros2012wave,rincon2016numerical} that for any initial data $\icz(x)\in \H_0^1(\Omega)$ and $\icf(x)\in L^2(\Omega)$, the solution exists for some time interval $[0,T_0)$, i.e. $u\in L^{\infty}(0,T_0;\H_0^1(\Omega))$. The existence of smooth solution for exponential wave equation was discussed in \cite{wang2015energy}.

Define an appropriate solution space as
\begin{equation}\label{eq:solsp}
    \V := \big\{ \chi \in \H^1(\Omega) \,\big|\, \text{supp}(\chi)\subseteq \Sigma \subset\subset \Omega \big\} \subset \H^1_0(\Omega).
\end{equation}
Then,
the Galerkin formulation of \eqref{eq:NLW} is derived by multiplying a test function $\chi\in \V$ on both sides followed by an integration over $\Omega$. Then, Green’s first identity leads to the problem of seeking a function $\uoft:=u(\cdot,\,t)\in \V$ satisfying the initial conditions \eqref{eq:NLWInit} and
\begin{equation}\label{eq:NLWGalerkin}
\IP{\dtt{u}(t)}{\chi}+ \IP{\nabla u(t)}{\nabla \chi}+
\IP{F'(u(t))}{\chi}=0, ~~\forall \chi\in \V,
\end{equation}
where $\IP{\cdot}{\cdot}=\IPOmega{\cdot}{\cdot}$ denotes the $L^2$ inner product in $\Omega$.
Moreover, the nonlinear wave equation \eqref{eq:NLW} subject to  zero Dirichlet boundary conditions  admits an energy conservation law with the first-order time-derivative of the Hamiltonian
\begin{equation}\label{eq:s3energyconservation}
0= \dot{E}(t;u):= \frac{d}{dt}  
\int_{\Omega}\left[\frac{1}{2}(\dt{u}(t))^2+\frac{1}{2}|\nabla u(t)|^2+F(u(t))\right]dx.
\end{equation}
It is higher desirable to have numerical schemes that inherit the conservation property.
In the literature, various approaches had been adopted for spatial discretization to obtain energy-conserving numerical schemes; commonly used ones include the finite difference method \cite{livu1995finite,matsuo2001dissipative}, the finite element method \cite{eidnes2018adaptive}, the Gauss-Legendre collocation method \cite{Reich}, the Fourier pseudo-spectral method \cite{chen2001multi,cai2013conservative,cai2016numerical,gong2017conservative}, the wavelet collocation method \cite{Zhu0}, the discontinuous Galerkin method \cite{liu2016hamiltonian} and so on.
Some of above methods, e.g., the Gauss-Legendre method and the Fourier pseudo-spectral method, can only preserve the energy on uniform grids. Nonuniform grids or scattered data sites are of particular importance for solving multi-dimensional problems with irregular domain.
Allowing the use of scattered data provides a possibility to develop adaptive energy-conserving schemes. For example, Yaguchi, Matsuo and Sugihara \cite{yaguchi2010extension,yaguchi2012discrete} proposed two different discrete variational derivative methods on non-uniform grids for certain classes of PDEs based on the mapping method (i.e. changing coordinates) and discrete differential forms. In addition, Eidnes, Owren and Ringholm \cite{eidnes2018adaptive} used weighted finite difference methods and partition of unity methods to construct adaptive energy preserving numerical schemes on non-uniform grids.

Meshless methods are popular in the past few decades for solving PDEs   \cite{belytschko1996meshless},  especially radial basis function and kernel-based methods   \cite{fasshauer2007meshfree,Wendland}  since they do not require generating a mesh and can excellently approximate multivariate function sampled on scattered nodes. Recently, the meshless Galerkin method have attracted some interest since it combines the advantages of meshless method and the finite element method.  Wendland firstly established the theory of RBFs with the field of Galerkin methods to solve PDEs in \cite{wendland1999meshless}. Afterwards, the meshless Galerkin method was applied to solve other types of equations, such as Schr\"{o}dinger equation \cite{kormann2013galerkin}, elliptic and parabolic PDEs on spheres \cite{legia04galerkin,narcowich2017novel,kunemund2019high}, and non-local diffusion equation \cite{lehoucq2016meshless}.
For conservation schemes on scatted nodes, Wu and Zhang \cite{WuZhang} proposed a meshless conservative method based on radial basis function (RBF) approximation for solving linear mutivariate Hamiltonian equation. The nonlinear Hamiltonian equation was recently studied by Sun and Wu in \cite{SunWu}; both papers dealt with conservation of sympleciticy rather than energy conservation.
As of the day of writing, there isn't any meshless Galerkin energy conservative schemes available.
With this premise,
the main contribution of this work is to provide a new kernel-based \emph{energy conservative} method that solves nonlinear wave equations \pde along with  rigorous convergence proofs.

Our analysis follows the standard approach in the method of lines, in which \pde were converted into a  Galerkin  formulation in some kernel-based approximation space in \sref{Sec:Radial basis functions and approximation spaces}. For the sake of energy conservation,
we define two projection operators and utilize them in the Galerkin formulation. In \sref{Sec:3} and \ref{sec:Energy conservation}, we show that the semi-discretized Galerkin solution is convergent and conserves energy in a discrete sense, respectively. Fully discretized solution is obtained by solving the ordinary differential equations (ODEs) from the Galerkin formulation by a energy-conserving method. In \sref{sec:numerical exmp}, a linear wave and Sine-Gordon equations in 1D were used to validate our theoretical convergence rate and energy conservation property. We conclude the paper with simulations of a 2D nonlinear wave equation.

%
%
%

\section{Kernel,  trial space, and projections}\label{Sec:Radial basis functions and approximation spaces}

This section contains some mathematical preparations for discretizing  \eref{eq:NLWGalerkin} spatially.
Firstly, we suppose that the domain, initial data, and nonlinearity are sufficiently smooth so that Hamiltonian equation \pde admits a unique solution $u^*$ with sufficient regularity to allow higher order methods to work. In particular,  we assume
\begin{equation}\label{eqUm}
u^*(\cdot,\,t)\in \V\cap \solHm(\Omega)=:\Vm
\end{equation}
for some integer $\solm>d/2$ and all $t\in[0,T]$.
Because of the condition $\solm>d/2$ and Sobolev embedding theorem, we are dealing with  continuous functions.

For the same order of smoothness $\solm$, let
$\Phim:\R^d \rightarrow\R$  be a \emph{symmetric positive definite (SPD)} reproducing kernel of $\H^\rbfm(\R^d)$ that  satisfies
\begin{equation}\label{eq:kernelproperty}
c_1(1+\|\omega\|^2)^{-\rbfm}\leq \widehat{\Phim}(\omega) \leq c_2(1+\|\omega\|^2)^{-\rbfm}
\mbox{ \quad for all $\omega\in\R^d$},
\end{equation}
for some  constants $0<c_1\leq c_2$.
Suppose that the bounded domain $\Sigma$ has $C^\rbfm$-smooth boundary and satisfies an interior cone condition, then,  the reproducing kernel Hilbert space of $\Phim$ on $\Sigma$, a.k.a. the native space \cite{Wu+Schaback-Locaerroestiradi:93},
is also norm-equivalent \cite{Buhmann,Wendland} to $\rbfHm(\Sigma)$.
The standard Whittle-Mat\'{e}rn-Sobolev kernel is a commonly used Hilbert spaces reproducing kernel that satisfies \eqref{eq:kernelproperty} exactly with $c_1=1=c_2$.
Because of the compactly supportedness in \eref{eq:AssumOfu}, we are motivated to use compactly supported kernels. Later in this paper, we use the family of piecewise polynomial Wendland functions\cite{Wendland-Soboerroestiinte:97} with some the shape parameter $\epsilon>0$, $d\leq s$, and $k=0,\ldots,3$ to define $C^{2k}$-kernels $\Phim(\cdot)=\phi_{s,k}(\epsilon \|\cdot\|) :\R^s\to\R$ of smoothness order
$$m=s/2+k+1/2 - (s-d)/2 = d/2+k+1/2.$$
For cases of  $d<s$, readers can refer to the theories of restricted kernels in \cite{Fuselier+Wright-ScatDataInteEmbe:12}.
For example, we will use $s=1,3$ and $k=1,2$ in our numerical computations in  \sref{sec:numerical exmp}, i.e.,
\begin{equation}\label{phis1}
  \phi_{s,1}(r)=(1- r)_{+}^{\ell+1}\big((\ell+1) r+1\big),
\end{equation}
and
\begin{equation}\label{phis2}
  \phi_{s,2}(r)=(1- r)_{+}^{\ell+2}\big((\ell^2+4\ell+3) r^2+(3\ell+6) r+3\big),
\end{equation}
where $\ell:=\lfloor s/2\rfloor +k +1$ and $(x)_{+}:=\max\{0,x\}$.


Using some set  $X = \{x_1,\ldots,x_N\}\subset\Sigma$ of trial centers, we construct finite dimensional trial spaces associated with a compactly supported SPD kernel $\Phim$
by
\begin{equation}\label{eq:RBFsubspace}
\trialsp:=\text{span}\big\{\Phim(x-x_1),\ldots,\Phim(x-x_N)\big\} \subset\rbfHm(\Sigma).
\end{equation}
For any trial function $v \in \trialsp$, we must have
\begin{equation}\label{eq:Omega-def}
 \text{supp}(v) \subseteq \Sigma_{\text{ext}}:= \Big\{ x\in\R^d\,\big|\, \text{dist}( x, \Sigma) < \frac12\text{diam}\big(\text{supp}(\Phim)\big) \Big\} \subseteq \Omega.
\end{equation}
Thus, we can specify a kernel-dependent domain $\Omega$ with sufficient smoothness for the application of RBF interpolation theories.
From here on when no confusion arises, we allow trial functions $v\in \trialsp$, which are already defined in $\R^d$, to be evaluated outside $\Sigma$, i.e., by an natural (or isometric) extension operator \cite[Thm. 10.46]{Wendland} of reproducing kernel Hilbert space. Then, the (time-independent) solution space $\Vm$ is contained in the (extended) native space $\rbfHm(\Sigma)$ of $\Phim$.

For any time-independent $f\in \Vm$, it is also in $\rbfHm(\Sigma)$. We let the unique interpolant of $f$ on $X\subset\Sigma$ from the approximation space $\trialsp$ be
\begin{equation}\label{eq:Interpolant}
\interp f=\sum_{j=1}^N{\lambda}_j\Phim(x-x_j) \in \trialsp.
\end{equation}
The approximation power of \eref{eq:Interpolant} in domain $\Sigma$ were studied extensively. We now restate the results in our contexts, i.e., convergence in $\Omega$ of trial functions with centers in $\Sigma\subset\Omega$.

\begin{lemma}\label{lem:InterpProp}
For any compact set $\Sigma\subset\R^d$, let $\Omega\supset\Sigma_{\text{ext}}$ as in \eref{eq:Omega-def} be bounded and satisfies an interior cone condition.
For some smoothness order $\rbfm>d/2$,
let $\Phim$ be a compactly supported symmetric positive definite kernel satisfying \eref{eq:kernelproperty}.
Denote by $\interp f\in\trialsp$ as in (\ref{eq:Interpolant}) the (natually extended) interpolant on $X\subset\Sigma$ to a function $f\in\Vm$. Then, for all sufficiently dense $X$ with fill distance
\[
 h_X := h_{X,\Sigma}= \sup_{\zeta\in \Sigma} \inf_{\eta\in X} \|{\zeta}-{\eta}\|_{\ell^2(\R^d)} ,
\]
the estimate
\begin{equation}\label{eq:conv est}
  \|   f-\interp f\|_{\H^j(\Omega)}\leq Ch_X^{\solm-j}\|  f\|_{\solHm(\Omega)}
\end{equation}
holds for $0\leq j \leq \solm$ with some generic constant $C$ independent of $f$.
\end{lemma}
\begin{proof}
 We reuse the convergence proof \cite[Thm. 4.1]{wendland1999meshless} of kernel interpolation in $\Sigma$,
 in which all analysis were carried out in $\R^d$ (either by power function or by orthogonality of interpolants). Estimates were brought back to $\Sigma$ essentially based on the continuity of the Sobolev extension involved (due to smoothness of $\Omega$ and $\Phi_m$.)
By construction, the Sobolev extension $E_{{\Sigma \to\R^d}}$ of $f\in{\Vm}$, whose support is strictly in $\Sigma$ by \eref{eq:AssumOfu}, \eref{eq:solsp}, and \eref{eqUm},  becomes the zero extension. Using  natural extension on the interpolant, we have
\[
 \|   f-\interp f\|_{\H^j(\Omega)} =  \|  E_{{\Sigma \to\R^d}} f-\interp f\|_{\H^j(\Omega)}
 \leq  \|  E_{{\Sigma \to\R^d}} f-\interp f\|_{\H^j(\R^d)},
\]
which brings us back to the same analysis in $\R^d$.
Therefore, $\|E_{{\Sigma \to\R^d}} f\|_{\H^m(\R^d)} \leq C \|f\|_{\H^m(\Sigma)}
\leq C' \| f \|_{\H^m(\Omega)}$ firstly by continuity of $E_{{\Sigma \to\R^d}}$ (which is indeed an equality) and then a trivial bound. 
%
\end{proof}

We now have the finite dimensional subspace $\trialsp$ ready for spatial discretization in the Galerkin method. To project functions $u$ from $\Vm$ to $\trialsp$, we define the $L^2$-orthogonal projection $\Po$
  by
\begin{equation}\label{eq:NLWProjection}
    \IP{\Po u}{\chi}
    =\IP{u}{\chi}\mbox{\quad for all }\chi\in\trialsp,
\end{equation}
and Ritz projection $\RN$  by
\begin{equation}\label{eq:ProjOpProp}
\IP{\nabla \RN u}{ \nabla \chi}=\IP{\nabla u}{ \nabla \chi}
\mbox{\quad for all } \chi\in\trialsp.
\end{equation}
These operators will be used on the nonlinear term $F'$, and initial conditions $\icz$, and $\icf$, in \pde respectively.
In the followings, we will discuss the error estimate of two projection operators. Though  \lref{lemma:Projection} about projection operator $\mathcal{P}$ is not needed in the analysis of this work, we present it here for completeness.

\begin{lemma}\label{lemma:Projection}
Suppose the assumptions in \lref{lem:InterpProp} hold. Let $I$ be a bounded interval of $\mathbb{R}$.  Suppose that a nonlinear smooth function $F':I\rightarrow\mathbb{R}$ satisfy $F'(0)=0$. Then for any trial function $\chi\in\trialsp$ with $\text{Range}(\chi)\subseteq I$, we have
\begin{eqnarray*}
    \|\Po F'(\chi)-F'(\chi)\|_{L^2(\Omega)}
    &\leq& C h_X^{\solm} \|  F'(\chi) \|_{\solHm(\Omega)}
    \\
    &\leq&   Ch_X^{\solm} (1+\|\chi\|_{L^{\infty}(\Omega)})^m \|\chi\|_{\solHm(\Omega)}    ,
\end{eqnarray*}
for some constant $C$
independent of  $h_X$ and $\chi$.

\end{lemma}
\begin{proof}
Let $f:=F'(\chi)$. We have $f\in \Vm$ by assumptions and let  $s_f:=\interp f\in\trialsp$ denote its unique interpolant.
The projection relation \eqref{eq:NLWProjection} and the fact that $\Po f -s_f \in \trialsp$  yields
\begin{eqnarray*}
  \IPOmega{ \Po f - f}{ \Po f - f}
  \leq \|\Po f - f\|_{L^2(\Omega)} \|s_f - f\|_{L^2(\Omega)}.
\end{eqnarray*}

Now we can finish the proof by \lref{lem:InterpProp}
\begin{eqnarray*}
  \|\Po f - f\|_{L^2(\Omega)}
  \leq \|s_f - f\|_{L^2(\Omega)}
&\leq& C h_X^\solm \| f\|_{\solHm(\Omega)},
\end{eqnarray*}
and  $\| f\|_{\solHm(\Omega)} = \|F'(\chi)\|_{\solHm(\Omega)} \leq C' (1+\|\chi\|_{L^{\infty}(\Omega)})^m\|F''\|_{W^{m,\infty}(I)} \|\chi\|_{\solHm(\Omega)}$,
which follows from \cite[Prop. 1.4.8]{danchin2005fourier} (with $p=r=2$, $s=\sigma=m$, and $F'=f$).

\end{proof}

\begin{lemma}\label{lemma:RitzProjection}
Suppose the assumptions in \lref{lem:InterpProp} hold.  Then for any function $u\in\Vm$ with a smooth boundary $\Omega$, then
\[
    \|\nabla(\RN u - u)\|_{L^2(\Omega)}\leq C h_X^{\solm-1}\|u\|_{\solHm(\Omega)}
\]
holds for some  constant $C$ independent of $u$. In cases when $d\leq 3$ and $X$ is quasi-uniform, we also have
\[
      \|\RN u - u\|_{L^2(\Omega)}\leq C_{\rho} h_X^{\solm}\|u\|_{\solHm(\Omega)}
\]
for another constant $C_{\rho}$ depending on the mesh ratio of $X$ but independent of $u$.
\end{lemma}

\begin{proof}
All constants $C$ below are generic.
The definition of $\RN$ in  \eqref{eq:ProjOpProp} suggests that
$
 \IP{\nabla(\RN u-u)}{\nabla \chi}=0\mbox{\quad for any }\chi\in \trialsp.
$
If we use $\chi = \interp u \in\trialsp$, then
        \begin{equation*}
        \begin{aligned}
        \big\|\nabla(\RN u-u)\big\|_{L^2(\Omega)}^2
        =&\IP{\nabla(\RN u-u)}{\nabla (\RN u-\chi+\chi-u)}\\
        =&\IP{\nabla(\RN u-u)}{\nabla (\chi-u)}\\
        \leq & \big\|\nabla(\RN u-u)\big\|_{L^2(\Omega)}\,\big\|\chi-u\big\|_{\H^1(\Omega)}
        \end{aligned}
        \end{equation*}
Applying \lref{lem:InterpProp} gives an estimate for the $\H^1$-seminorm of $\RN u-u$:
\begin{equation}\label{eq:RitzProjEsti}
\big\|\nabla(\RN u-u)\big\|_{L^2(\Omega)}
\leq \big \| \interp u -u \big\|_{\H^1(\Omega)}
\leq C h_X^{\solm-1}\|u \|_{\solHm(\Omega)} .
\end{equation}
Poincar\'{e} inequality allows us to bound $\big\|\RN u-u\big\|_{L^2(\Omega)}\leq C \big\|\nabla(\RN u-u)\big\|_{L^2(\Omega)}$ by the same estimate.
To do better, consider the Poisson equation
\begin{equation}
-\Delta \psi = \varphi =: \RN u -  u \text{ in } ~\Omega,
\qquad
\psi=0\text{ on } \partial\Omega,
\end{equation}
as in the proof of \cite[Thm. 1.1]{thomee2006Galerkin}.
With the given smoothness in PDE data, we know the Poisson solution $\psi$ exists.
We now use this solution $\psi$ to yield
\begin{eqnarray*}
  \big\|\RN u-  u\big\|_{L^2(\Omega)}^2
  &=&  \IP{\varphi}{\varphi}=  \IP{\varphi}{-\Delta \psi}
  =  \IP{\nabla\varphi}{\nabla \psi}.
\end{eqnarray*}
By applying \eref{eq:ProjOpProp} with $\chi =\interp \psi \in\trialsp$, we have
\begin{eqnarray}
    \big\|\RN u-  u\big\|_{L^2(\Omega)}^2
    &=& \IP{\nabla \varphi}{\nabla \psi - \nabla \chi}
    \nonumber \\
    &\leq&  \big\|\nabla(\RN u-  u)\big\|_{L^2(\Omega)} \big\|\nabla(\psi-\chi)\big\|_{L^2(\Omega)}
    \nonumber  \\
    &\leq& C   h_X^{\solm-1}\|u \|_{\solHm(\Omega)}  \big\| \psi -\interp \psi \big\|_{\H^1(\Omega)},
    \label{eq:pre-headach}
\end{eqnarray}
where the last inequality is obtained by applying \eref{eq:RitzProjEsti}. To complete the proof, we treat $\psi$ as a $\H^2(\Omega)$ function that is outside the native space.  By imposing $2>d/2$ and $X$ being quasi-uniform with bounded mesh ratio,
the interpolation error estimate
\cite[Thm. 4.2]{Narcowich+WardETAL-SoboErroEstiBern:06}
can be deployed, due to similar arguments as in the proof of \lref{lem:InterpProp},
to get
\begin{equation}\label{eq:headache}
 \big\| \psi -\interp \psi \big\|_{\H^1(\Omega)}
\leq
C_\rho h_X \big\| \psi\big\|_{\H^2(\Omega)}.
\end{equation}
Combining \eref{eq:pre-headach}, \eref{eq:headache}, and a regularity estimate of Poisson equations, i.e.,
\[
\big\| \psi\big\|_{\H^2(\Omega)} \leq C \big\| \varphi \big\|_{L^2(\Omega)}= C \big\| \RN u-  u \big\|_{L^2(\Omega)},
\]
yields the desired $L^2$-estimate of $\RN u-u$.
\end{proof}

\section{Semi-discretize solution and error estimates}
\label{Sec:3}

Using the finite dimensional subspace in \sref{Sec:Radial basis functions and approximation spaces}, we propose to discretize the Galerkin formulation \eqref{eq:NLWGalerkin} by seeking
the solution
\begin{equation}\label{eq:solutionRepresention}
    u_X(x,t)=\sum_{j=1}^N \alpha_{j}(t)\,\Phi_m(x-x_j) \in \trialsp
\end{equation}
with time dependent coefficients $\alpha_{j}\in C^2([0,T])$
for all $j =1,\ldots,N$, that solves the discretized equation
\begin{equation}\label{eq:NLWGalerkinDiscrete2}
    \IP{\dtt{u}_X}{\chi}+\IP{\nabla u_X}{\nabla \chi}+\IP{\Po F'(u_X)}{\chi}=0
    \mbox{\quad for all } \chi\in \trialsp,
\end{equation}
subject to initial conditions
\begin{equation}\label{eq:NLWGalerkinDiscrete2IC}
    u_X(\cdot, 0) = \RN \icz
    \mbox{\quad and \quad}
    \dt{u}_X(\cdot, 0) = \RN \icf
    \mbox{\quad in }  \Omega.
\end{equation}
We remark that \eref{eq:NLWGalerkinDiscrete2} differs from standard Galerkin approach in \cite{kunemund2019high} by an additional $\Po$ projection of the nonlinearity, whose purpose is energy conservation that will be discussed in the coming section. For now, we study the convergence behaviour of \semipde~    that agrees with well-known finite element results on Galerkin approximation of second order hyperbolic equations.

\begin{theorem}\label{thm:NLWSemidiscrete}
Let $u^* \in C^2([0,T];\Vm)$ denote the  exact solution  to \pde and $u_X(t)\in\trialsp$ with $C^2([0,T])$ coefficients solves \semipde.
Suppose the assumptions in \lref{lem:InterpProp}--\ref{lemma:RitzProjection} hold for $m\geq1$ and $m>d/2$.
Then there exist constants $C$ independent of $t$, $h_{X}$, and $u_X$ such that
    \begin{equation}\label{eq:ErrEst1}
    \big\|u^*(\cdot, t)-u_X(t)\big\|_{L^2(\Omega)} \leq Ch_{X}^{\solm-j}\| u^*\|_{\H^2([0,T];\solHm(\Omega))},
    \end{equation}
holds for $j=0$  in any dimension $d\leq 3$ and for $j=1$ otherwise,
    and
    \begin{equation}\label{eq:ErrEst2}
    \big\|\nabla u^*(\cdot, t)-\nabla u_X(t)\big\|_{L^2(\Omega)} \leq Ch_{X}^{\solm-1} \| u^*\|_{\H^2([0,T];\solHm(\Omega))}.
    \end{equation}
hold for any $t\in[0,T]$.
\end{theorem}
\begin{proof}
For $t\in[0,T]$, we analyze the error function in parts:
\[
u(\cdot,t)-u_X(t)=\underbrace{u(\cdot,t)-\RN u(\cdot,t)}_{\displaystyle:=\eta(t)}+\underbrace{\RN u(\cdot,t)- u_X(t)}_{\displaystyle:=\theta(t)}.
\]
Convergence estimates of $\eta(t)$ were done in \lref{lemma:RitzProjection} and we only need to work on $\theta(t)$.
Subtracting  \eqref{eq:NLWGalerkinDiscrete2} from \eqref{eq:NLWGalerkin} yields
\begin{eqnarray}
\IP{\dtt{\eta}+\dtt{\theta}}{ \chi}+
 \IP{\nabla\eta+\nabla\theta}{ \nabla \chi}
&=&\IP{\Po F'(u_X)-F'(u)}{\chi} \nonumber
\end{eqnarray}
for all $\chi\in\trialsp$.  Using definitions \eref{eq:ProjOpProp} and \eref{eq:NLWProjection} of $\Po$ and $\RN$ can further simplify the above to
\begin{eqnarray}
\IP{\dtt{\eta}+\dtt{\theta}}{ \chi}+ \IP{\nabla\theta}{ \nabla \chi} &=& \IP{F'(u_X)-F'(u)}{\chi}
\mbox{\quad for all }\chi\in\trialsp
\label{eq:ErrEq1_new}
\end{eqnarray}
which coincides with the finite elements counterparts.  Now that $\Po$ disappear; one can conclude that using $\Po$ in \eref{eq:NLWGalerkinDiscrete2} does not affect convergence.

Using the Lipschitz continuity of $F'$ and following the same nontrivial manipulations as in proof of  \cite[Thm. 1]{Dupont-EstiGalerMeth:73} from equations (2.8) to (2.10) there, one can obtain an estimate that
\begin{equation}\label{eq:GronwallEst1}
    \|\dt{\theta}(t)\|^2_{L^2(\Omega)}+\|\nabla\theta(t)\|^2_{L^2(\Omega)}
\leq C_T\int_{0}^{t}\Big(\|\dtt{\eta}(s)\|_{L^2(\Omega)}^2+  \|\eta(s)\|_{L^2(\Omega)}^2\Big)ds
\end{equation}
for $0\leq t\leq T$ and some constant $C_T$ depending on $T$.
We use \lref{lemma:RitzProjection} to bound \eref{eq:GronwallEst1} further and obtain
\begin{eqnarray}
    \|\dt{\theta}(t)\|_{L^2(\Omega)}+\|\nabla\theta(t)\|_{L^2(\Omega)}
    &\leq& Ch_X^{\solm}\left(\int_{0}^{t}\Big(\|\dtt{u}(\cdot,s)\|_{\solHm(\Omega)}^2+\|u(\cdot,s)\|_{\solHm(\Omega)}^2 \Big) ds\right)^{1/2} \nonumber
    \\
    &\leq& C h_X^{\solm} \| u \|_{\H^2([0,t];\solHm(\Omega))}
    \mbox{\quad for }t\in[0,T].
    \label{eq:GronwallEst3}
\end{eqnarray}
Maximizing the upper bound with $t=T$ proves \eref{eq:ErrEst2}. To finish the proving of \eref{eq:ErrEst1}, we use
\[
    \|\theta(t)\|
    \leq\|\theta(0)\|+\int_{0}^{t}\|\dt{\theta}(s)\|ds
    \leq\|\theta(0)\|+t\max_{0\leq s\leq t}\|\dt{\theta}(s)\|.
\]
In both estimates, we see that convergence rates were indeed limited by \lref{lemma:RitzProjection}, but \eref{eq:GronwallEst3} imposes higher regularity requirement  in time to the solution.
\end{proof}

\subsection{Some linear algebras}\label{sec:LA}
We complete this section by expressing the semi-discretized system \semipde~ in matrix form.
The solution is in the form of \eref{eq:solutionRepresention}, or in compact notations
\begin{equation}\label{eq:trial sol}
    u_X(\cdot,t)= \Phim(\cdot,X)  \balpha(t)  \in \trialsp
    \mbox{\quad for some } \balpha(t) \in \big(C^2([0,T])\big)^N,
\end{equation}
in which we allow the kernel to operate on sets, i.e., on $X=\{x_1,\ldots,x_{|X|}\}$ and $Z=\{z_1,\ldots,z_{|Z|}\}$, to yield matrices
\[
    \Phim(Z,X) \in \R^{|Z|\times|X|}
    \mbox{\quad with entries }[\Phim(Z,X)]_{ij}= \Phim(z_i-x_j).
\]
Next, we shall obtain the matrix form of equation \eqref{eq:NLWGalerkinDiscrete2} by using $\chi=\Phim(\cdot-x_k)$ for $k=1,\ldots,N$. Handling the linear terms are straight forwarded.
By using the definition of projection operator \eqref{eq:NLWProjection} and 
the vector function $\bvartheta$ of time is defined by entries $[\bvartheta]_j= \IPOmega{ F'(u_X)}{\Phim(\cdot-x_j)}$ ($j=1\ldots,|Z|$),
the semi-discrete equation  \eqref{eq:NLWGalerkinDiscrete2} becomes
\begin{equation}\label{eq:NLWGalerkinDiscrete3}
\mtxA \dtt{\balpha}(t)+\mtxB\balpha(t)+
\bvartheta 
=0,
\end{equation}
for two time-independent Gramian matrices of size  $N\times N$ with entries
\[
[\mtxA]_{jk}= \IPOmega{\Phim(\cdot-x_j)}{\Phim(\cdot-x_k)},
\]
and
\begin{eqnarray*}
[\mtxB]_{jk} &=& \IPOmega{\nabla\Phim(\cdot-x_j)}{\nabla\Phim(\cdot-x_k)}
\\ &=&   
\sum_{\scriptsize \begin{array}{c}
   |\xi|\!=\!1, 
   \xi\!\in\!\mathbb{N}_0^d
 \end{array}}
\IPOmega{D^{\xi}\Phim(\cdot-x_j)}{D^{\xi}\Phim(\cdot-x_k)}.
\end{eqnarray*}
Note, when deriving \eref{eq:NLWGalerkinDiscrete3}, the  projected nonlinearity $\mathcal{P}F'(u_X)\in\trialsp$ at some time $t$ takes the expansion
\begin{equation}\label{eq:projectionRepresention}
     \Po F'(u_X)(\cdot,t) 
    = \Phim(\cdot,X)\Big(
    \underbrace{ (\mtxA)^{-1} \bvartheta }_{=: \bgamma(t)}
    \Big)
\end{equation}
with the kernel coefficient vector $\bgamma$ depending implicitly  on $u_X$ and/or $\balpha$.


Lastly, we need to work out the matrix formulas for the initial conditions.
By the definition of Ritz projection, we obtain formulas for initial conditions of $\balpha$ and $\dt{\balpha}$,
\[
    \balpha(0) = (\mtxB)^{-1}\IP{\nabla\icz}{\nabla\Phim(\cdot-x_j)},
\]
\[
    \dt{\balpha}(0) = (\mtxB)^{-1}\IP{\nabla\icf}{\nabla\Phim(\cdot-x_j)}.
\]
Clearly, $\nabla\Phim(\cdot-x_j)$  for $x_j\in X$ forms a set of independent functions. Thus, the Gramian matrix $\mtxB$ is positive definite and invertible.
%

\section{Energy conservation}\label{sec:Energy conservation}

Let $\dt{\balpha}(t) = \bbeta(t)$ and we rewrite {\semipde} as a system of first order differential equations:
\begin{equation}\label{eq:NLWsemidiscrete}
\left\{
  \begin{array}{lcll}
    \dt{\balpha} &=& \bbeta, & t\in[0,T]\\
    \dt{\bbeta} &=& -(\mtxA)^{-1}\mtxB \balpha -\bgamma, & t\in[0,T]\\
    \balpha(0) &=& \balpha_0, &\\
    \bbeta(0) &=& \bbeta_0. &\\
  \end{array}
\right.
\end{equation}
We now verify the energy conservation property of the semi-discrete solution $u_X$ in \eref{eq:trial sol} in the sense of  \eref{eq:s3energyconservation}. Using the coefficient vector $\balpha$, we got
\begin{equation}\label{eq:NLWDisctEnergy}
\begin{aligned}
E(t;u_X) &= \int_{\Omega}\left(\frac{1}{2} (\dt{u}_X)^2+\frac{1}{2}|\nabla u_X|^2+F(u_X)\right)dx
\\
&=\frac{1}{2} \bbeta^T \big(\mtxA\big){\bbeta} + \frac{1}{2}{\balpha}^T\big(\mtxB\big){\balpha} +\IPOmega{F(u_X)}{1}.
\end{aligned}
\end{equation}
Differentiating the energy with respect to time and using symmetry in Gramian, we have
\[
\dt{E}(t;u_X)
= \bbeta^T \big(\mtxA\big)\dt{\bbeta} + {\balpha}^T\big(\mtxB\big)\dt{\balpha} +\IP{F'(u_X)}{\dt{u}_X}.
\]
If we use \eref{eq:NLWsemidiscrete} to remove all first order derivatives, we got a simpler expression
\begin{eqnarray*}
  \dt{E}(t;u_X)
    &=& -\bbeta^T \big(\mtxA\big)\bgamma +\IP{F'(u_X)}{\dt{u}_X}
    \\ &=& -\IP{\Po F'(u_X)}{\dt{u}_X}+\IP{F'(u_X)}{\dt{u}_X} =0,
\end{eqnarray*}
which is zero because of \eref{eq:NLWProjection}, and the fact that $\bbeta$ and $\bgamma$ are coefficient vectors of $\dt{u}_X$ and $\Po F'(u_X)$ in $\trialsp$. We conclude this by a theorem.

\begin{theorem}
The convergent semi-discretize solution $u_X\in\trialsp$ in Thm.~\ref{thm:NLWSemidiscrete} conserves energy in the sense of  \eref{eq:s3energyconservation}.
\end{theorem}

\subsection{Discrete energy conservation}

In practice, it is not straightforward to analytically evaluate integrals of the nonlinear term $\IPOmega{F'(u_N)}{\Phim(x-x_j)}$ in order to compute the coefficient vector $\bgamma$  in \eqref{eq:projectionRepresention}. Next, we move on to a discrete version of $\bgamma$ that is approximated by some quadrature formula.

In the finite dimensional subspace $\trialsp$, we employ the translates of a kernel at scattered nodes $X$. Moreover, we use some sufficiently dense set of quadrature points $Z\subset\Omega$ of size $M$ for numerical integration. Then we can approximate the $L^2$-inner product by a weighted inner product, namely
\begin{equation}\label{eq:QuadratureFormula}
    \IPOmega{f}{g}=\int_{\Omega}f(x)g(x)d x\approx \sum_{j=1}^{M}w_j f(z_j)g(z_j)=\boldsymbol{f}^T W\boldsymbol{g}
    =\IP{\boldsymbol{f}|_{Z}}{\boldsymbol{g}|_{Z}}_{W},
\end{equation}
where $W=\text{diag}(w_1,\ldots,w_M)$ denotes the diagonal matrix of quadrature weight $w_j$ for $j=1,\ldots,M$.
Using the quadrature formula \eqref{eq:QuadratureFormula}, we define the discrete energy as
\begin{equation}\label{eq:NLWEnergy}
\begin{aligned}
E^d(t;u_X)=\frac{1}{2}\bbeta^T{A}\bbeta+\frac{1}{2}\balpha^T{B}\balpha+\big(\mathbf{1}_{M},F({u_X})|_{Z} \big)_{W},
\end{aligned}
\end{equation}
where
\[{A}:=\big[\Phim(Z,X)\big]^TW\big[\Phim(Z,X)\big]
\approx \mtxA\]
and
\[{B}:=
\sum_{\scriptsize \begin{array}{c}
   |\xi|\!=\!1, 
   \xi\!\in\!\mathbb{N}^d_0
 \end{array}}
\big[ D^{\xi}\Phim(Z,X)\big]^T W \big[D^{\xi}\Phim(Z,X)\big]
\approx \mtxB.\]
Instead of \eref{eq:projectionRepresention} and \eref{eq:NLWsemidiscrete}, we now use
\begin{equation}\label{eq:gammad2}
 \boldsymbol{\Upsilon}(t) = {A}^{-1} \big[ \Phim(Z,X)\big]^T W \big[ F'({u_X})|_{Z}\big] \approx \bgamma
\end{equation}
to define another ODE via numerical quadrature \eref{eq:QuadratureFormula} as
\begin{equation}\label{eq:NLWsemidiscreteVers}
\left\{
  \begin{array}{lcll}
    \dt{\balpha} &=& \bbeta, & t\in[0,T]\\
    \dt{\bbeta} &=& -{A}^{-1}{B} \balpha -\boldsymbol{\Upsilon}, & t\in[0,T]
    \\
    \balpha(0) &=&
    \displaystyle
    \sum_{\tiny \begin{array}{c}
   |\xi|\!=\!1 , \xi\!\in\!\mathbb{N}^d_0
 \end{array}}
 {B}^{-1}  \big[D^{\xi}\Phim(Z,X)\big]^TW\big[D^{\xi}u_0(Z)\big],
 &\\
    \bbeta(0) &=&
    \displaystyle
    \sum_{\tiny \begin{array}{c}
     |\xi|\!=\!1, \xi\!\in\!\mathbb{N}^d_0
 \end{array}}
 {B}^{-1} \big[D^{\xi}\Phim(Z,X)\big]^TW\big[D^{\xi}u_1(Z)\big]. &\\
  \end{array}
\right.
\end{equation}
We can again prove that the discrete energy $E^d(t;u_X)$ is conserved by the semi-discrete solution of \eqref{eq:NLWsemidiscreteVers}.
\begin{theorem}
    Let the discrete energy $E^d(t;u_X)$ be defined by the equation \eqref{eq:NLWEnergy} and $\balpha$, $\bbeta$ satisfy the equations \eref{eq:gammad2}--\eqref{eq:NLWsemidiscreteVers}, then the energy $E^d(t;u_X)$ is conserved with respect to time.
\end{theorem}
\begin{proof}
    Differentiating the energy  with respect to $t$ and using the fact that ${A}$, ${B}$ are symmetric, we obtain
    $$\frac{d}{dt}E^d(t;u_X)=\bbeta^T{A}\dot{\bbeta}+\balpha^T{B}\dot{\balpha}+\sum_{j=1}^Nw_j F'(u_X(z_j,t))\dot{u}_{X}(z_j,t).$$
    Using the ODE \eqref{eq:NLWsemidiscreteVers}, we have
    \begin{equation*}
    \begin{aligned}
    \frac{d}{dt} E^d(t;u_X)&=
    \bbeta^T{A}\big(-{A}^{-1}{B}\balpha-   \boldsymbol{\Upsilon}  \big)
    +\balpha^T{B}\bbeta
    +\big[ F'(u_X(Z,t))\big]^TW\big[\Phim(Z,X)\big]\bbeta\\
    &=-\bbeta^T{A}   \boldsymbol{\Upsilon} 
    + \bbeta^T\big[ \Phim(Z,X)\big]^T W \big[ F'\big(u_X(Z,t)\big)\big]=0,
    \end{aligned}
    \end{equation*}
where the last equality comes from our definition of $ \boldsymbol{\Upsilon} $ in \eref{eq:gammad2}.
\end{proof}

We remark that the choice of quadrature weight $W$ has no importance in the conservation of the energy, but it still affects the error between the discrete approximate energy \eref{eq:NLWEnergy} to the continuous true one \eref{eq:NLWDisctEnergy}.
Numerical integrations in \eref{eq:NLWDisctEnergy}, i.e., the choice of  the weight matrix $W$,  were to be performed on some background regular grids, which is independent to scattered centers used to define \eref{eq:RBFsubspace}.
We have many standard choices of quadrature formulas, such as Simpson's rule, Gauss quadrature rule, the meshless quadrature formula \cite{WuZhang} and so on.
When they were applied to compactly supported waves, we can expect higher than theoretically predicted order of convergence.


\subsection{Energy conservation by fully discretized solution}
To define a  fully discretized solution, it remains to discretize the semi-discrete system in time.
We employ the average vector field (AVF) method \cite{quispel2008new}, which is an energy-conserving method for solving ODE systems.
Consider a first-order ODE in the form of
\begin{equation}\label{eq:ODEs}
\frac{dy}{dt}=f(y), \quad y\in\mathbb{R}^d.
\end{equation}
The second-order AVF method for $y_n:=y(t_n)\longmapsto y_{n+1}:=y(t_n+\tau)$ is given as
\begin{equation}\label{eq:AVFmethod}
\frac{1}{\tau} (y^{n+1}-y^n)=\int_{0}^1 f\big((1-\xi)y_n+\xi y_{n+1}\big)d\xi,
\end{equation}
where $\tau$ is the step size. Error estimate of AVF methods can be found in \cite{cai2016numerical}.
For the linear wave equation with $f(y)=cy$ in \eref{eq:ODEs}, one can see that the second order AVF method is just the midpoint or Crank-Nicolson scheme. For any linear autonomous system of ODEs, the second order AVF method coincides with the midpoint scheme.

We apply  AVF \eref{eq:AVFmethod} to the ODE system \eqref{eq:NLWsemidiscreteVers} to obtain
\begin{equation}\label{eq:NLWfulldiscrete2}
\begin{cases}
\displaystyle
\frac1{\tau}(\balpha^{n+1}-\balpha_u^n)= \frac12 ({\bbeta^n+\bbeta^{n+1}}),
\smallskip\\  
\displaystyle
\frac1{\tau}(\bbeta^{n+1}-\bbeta^n)=
-\frac12{A}^{-1}{B}({\balpha^n+\balpha^{n+1}})
\\ \qquad\qquad\qquad\qquad\qquad
-{A^{-1}\big[\Phim(Z,X)\big]^T W}
\left(\displaystyle\frac{F(u_X^{n+1})-F(u_X^n)}{u_X^{n+1}-u_X^n}\right)\Big|_{Z}.
\end{cases}
\end{equation}
Here in \eref{eq:NLWfulldiscrete2}, we again evaluate functions on sets, i.e., $F(u_X^n)|_{Z} := F\big(u_X(z_j,t^n)\big)$ is a  vector in $\R^M$ that implicitly depends on $\balpha^n$.


By eliminating the intermediate variable $\bbeta$, we recast the fully discretized system \eqref{eq:NLWfulldiscrete2} in terms of $\balpha$ as the following nonlinear ODE
\begin{equation}\label{eq:FullDisNLW2EquivForm}
\frac{\balpha^{n+1}-2\balpha^n+\balpha^{n-1}}{\tau^2}
=-{A}^{-1}{B}\frac{\balpha^{n+1}+2\balpha^n+\balpha^{n-1}}{4}
-{A^{-1}\big[\Phim(Z,X)\big]^T W}{\boldsymbol{q}}, 
\end{equation}
where the $\R^{|X|}\times\R^{|X|}\times\R^{|X|}\to\R^M$ vector function is defined by
\begin{eqnarray*}
{\boldsymbol{q}} = {\boldsymbol{q}}(\balpha^{n+1},\balpha^n,\balpha^{n-1})
&:=&\left( \frac{F(u_X^{n+1})-F(u_X^{n})}{u_X^{n+1}-u_X^{n}}
+\frac{F(u_X^{n})-F(u_X^{n-1})}{u_X^{n}-u_X^{n-1}}  \right)\Big|_{Z}.
\end{eqnarray*}
This implicit equation   can be solved by using iteration methods such as the fixed-point iteration method.


\begin{theorem}\label{thm:EnergyconservationFulldiscrete}
    Let $\balpha^n$ and $\bbeta^n$ be the solution to the ODE system in \eqref{eq:NLWfulldiscrete2}. Then the fully discretized  solution in the form of \eref{eq:trial sol} conserve the discrete energy in the sense \eqref{eq:NLWEnergy}. 
\end{theorem}
    \begin{proof}
        Multiplying $(\balpha^{n+1}-\balpha^n)^T {A}$ from the left to the second equation in \eqref{eq:NLWfulldiscrete2} yields
\begin{align*}
&\frac{1}{\tau}(\balpha^{n+1}-\balpha^n)^T  {A}(\bbeta^{n+1}-\bbeta^n)
\\
&=\frac{-1}{2}(\balpha^{n+1}-\balpha^n)^T {B}({\balpha^{n+1}+\balpha^n} ) -\big[\Phim(Z,X)(\balpha^{n+1}-\balpha^n)\big]^T W\left(\displaystyle\frac{F(u_X^{n+1})-F(u_X^n)}{u_X^{n+1}-u_X^n}\right)\Big|_{Z}.
\end{align*}
Since $\Phim(Z,X)(\balpha^{n+1}-\balpha^n)=(u_X^{n+1}-u_X^n)\big|_{Z}$, we have
\begin{eqnarray}
&&\frac{1}{\tau}(\balpha^{n+1}-\balpha^n)^T  {A}(\bbeta^{n+1}-\bbeta^n) \nonumber
\\
&&\qquad=\frac{-1}{2}(\balpha^{n+1}-\balpha^n)^T {B}({\balpha^{n+1}+\balpha^n} )
-\mathbf{1}^T W F({u_X^{n+1}})\big|_{Z}+\mathbf{1}^TWF({u_X^{n}})\big|_{Z}.
\label{Eq:InnerProduct}
\end{eqnarray}      
On the other hand, using the first equation in \eqref{eq:NLWfulldiscrete2}, we have
\begin{equation}\label{Eq:InnerProduct2}
  \frac{1}{\tau}(\balpha^{n+1}-\balpha^n)^T  {A}(\bbeta^{n+1}-\bbeta^n)
  =\frac{1}{2}(\bbeta^{n+1}+\bbeta^n)^T  {A}(\bbeta^{n+1}-\bbeta^n).
\end{equation}
By equalling the right-handed sides of \eref{Eq:InnerProduct} and \eref{Eq:InnerProduct2}, we can deduce
    \begin{align*}
        &\frac{1}{2}(\bbeta^{n+1})^T {A}\bbeta^{n+1}+\frac{1}{2} (\balpha^{n+1})^T {B}\balpha^{n+1}+\mathbf{1}^TW F({u_X^{n+1}})\big|_{Z}\\
        &\quad=\frac{1}{2}(\bbeta^{n})^T {A}\bbeta^{n}+\frac{1}{2} (\balpha^n)^T {B}\balpha^n+\mathbf{1}^TWF({u_X^{n}})\big|_{Z},
    \end{align*}
    or $E^d(t^{n+1},u_X) = E^d(t^{n},u_X)$,  which completes the proof.
    \end{proof}

\section{Numerical examples}\label{sec:numerical exmp}
Several numerical examples are presented to verify the accuracy and energy conservation of the proposed methods.

\subsection{Linear wave and spatial convergence}\label{exmp1}
We consider the linear wave equation ($F'(u)=0$) in $\Omega\times[0,T]=[-5,5]\times[0,1]$  subject to  initial conditions
\begin{equation}\label{eq:LW_initcond}
    u_0(x)=(1-x^2)_{+}^{\mu+1}\in C^{\mu}(\Omega),\quad u_1(x)=0,
\end{equation}
where $\mu\in\{1,2,4\}$, and $u_0(x)$ is a compactly supported function with finite smoothness. The exact solution is $u^*(x,t)=\frac{1}{2}\big(u_0(x-t)+u_0(x+t)\big)$ by the D'Alembert formula and $\text{supp}(u^*):=\Sigma\subset\Omega$ for all $t\in[0,1]$.

We deploy  four unscaled ($\epsilon=1$) Wendland's compactly supported functions $\phi_{s,k}$, whose supports are all of radius 1, with $(s,k)=(1,1)$, $(3,1)$, $(1,2)$, and $(3,2)$ in the ODE  \eref{eq:NLWfulldiscrete2} for the coefficients for the fully discretized solution. The first two functions \eref{phis1} result in a $\H^m$-reproducing kernels for  $m=2$ and the latter two \eref{phis2} yield that for $m=3$. Trial centers $X$ are uniformly spaced in $ [-4,4]$,
so that all trial functions satisfy zero Dirichlet boundary condition on $\partial\Omega$ automatically.
For this 1D problem, we take the quadrature points $Z$ and weight $W$ via an adaptive integration scheme\footnote{Adaptive integrations were done by the   MATLAB built-in function \texttt{integral}.}  with $1E-6$ and $1E-10$     relative and absolute error tolerance.
We remark that
the quadrature formula only affects the error between continuous energy and discrete energy, but doesn't affect the energy conservation.
To test the spatial discretization error, the AVF method is utilized with a small time step $\tau=1E-5$.

\fref{fig:LLconv}(a) and (b) show the $L^2$-error profile (that is estimated by $2048$ nodal values) against shrinking $h_X$ of sets of uniform RBF centers in $\Sigma$ for solving linear wave equation subject to the $C^2$-smooth initial conditions \eref{eq:LW_initcond} with $\mu=1,2$. By design, the (effective) smoothness order in Thm.~\ref{thm:NLWSemidiscrete} is limited by the initial conditions, i.e., $m\leq\mu$. Hence, we see that all four tested kernels yield similar performance in terms of accuracy. Overall, we also observe a superconvergence rate in our numerical results, i.e., the actual convergence rate is about the double of the predicted convergence rate. This phenomenon has been observed and studied by  Schaback in \cite{schaback2018superconvergence}.
One should see Thm.~\ref{thm:NLWSemidiscrete} as a theoretically guaranteed minimum convergence rate.

\begin{figure}[t]
    \centering
    \subfigure[$L^2$-error,
            $\mu=1$]{\includegraphics[width=.32\textwidth]{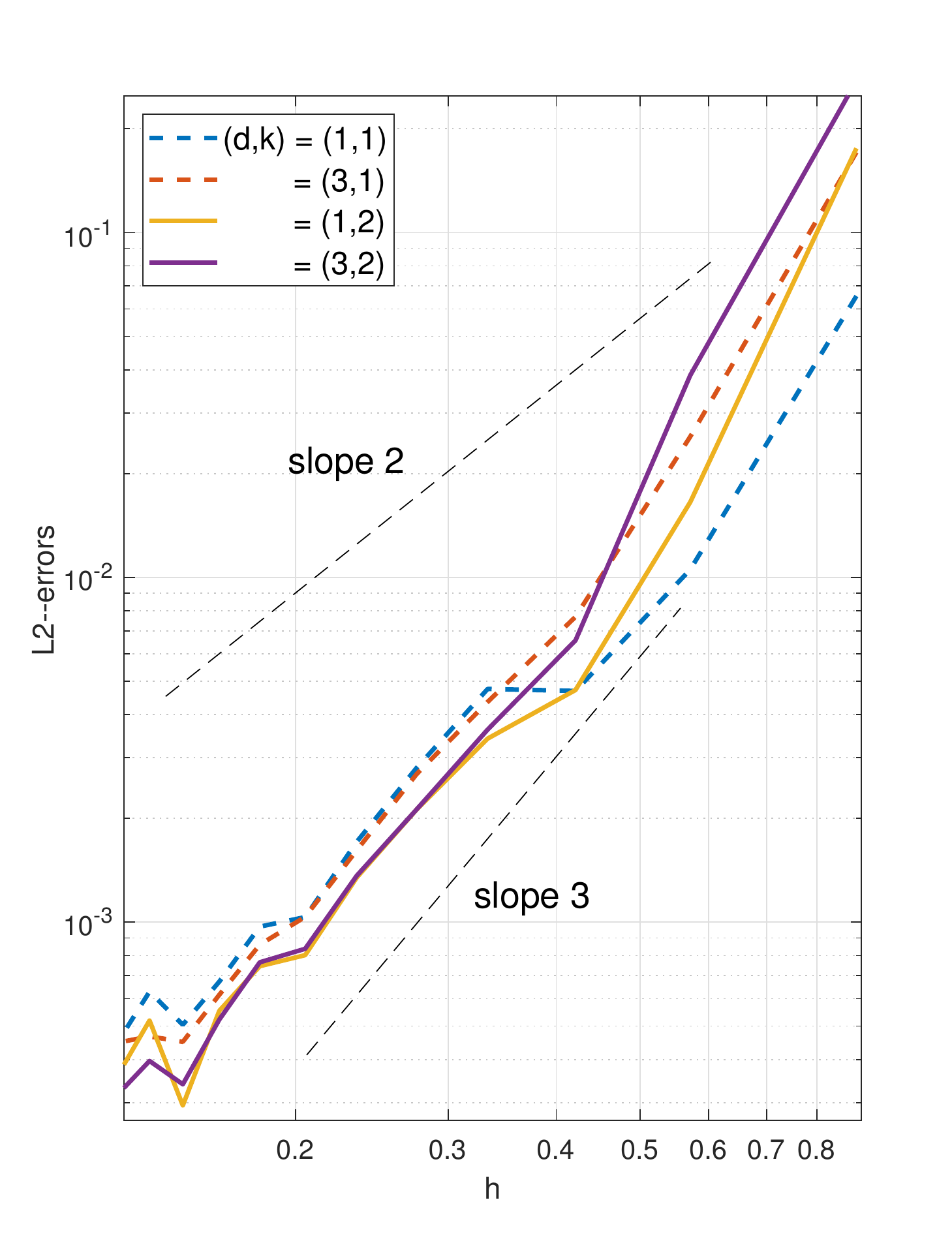}}
    \subfigure[$L^2$-error,
            $\mu=2$]{\includegraphics[width=.32\textwidth]{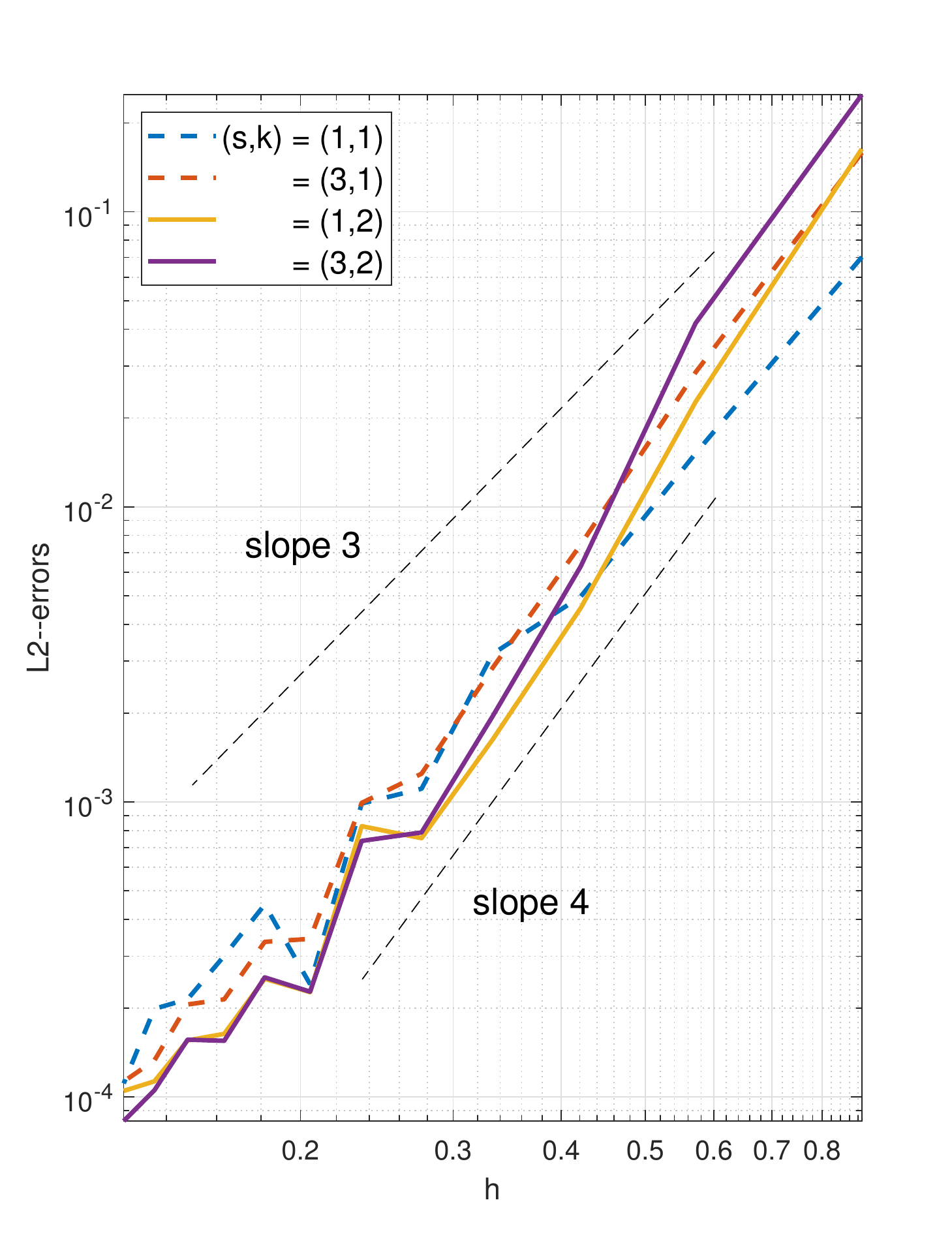}}
    \subfigure[$L^2$-error,
            $\mu=4$]{\includegraphics[width=.32\textwidth]{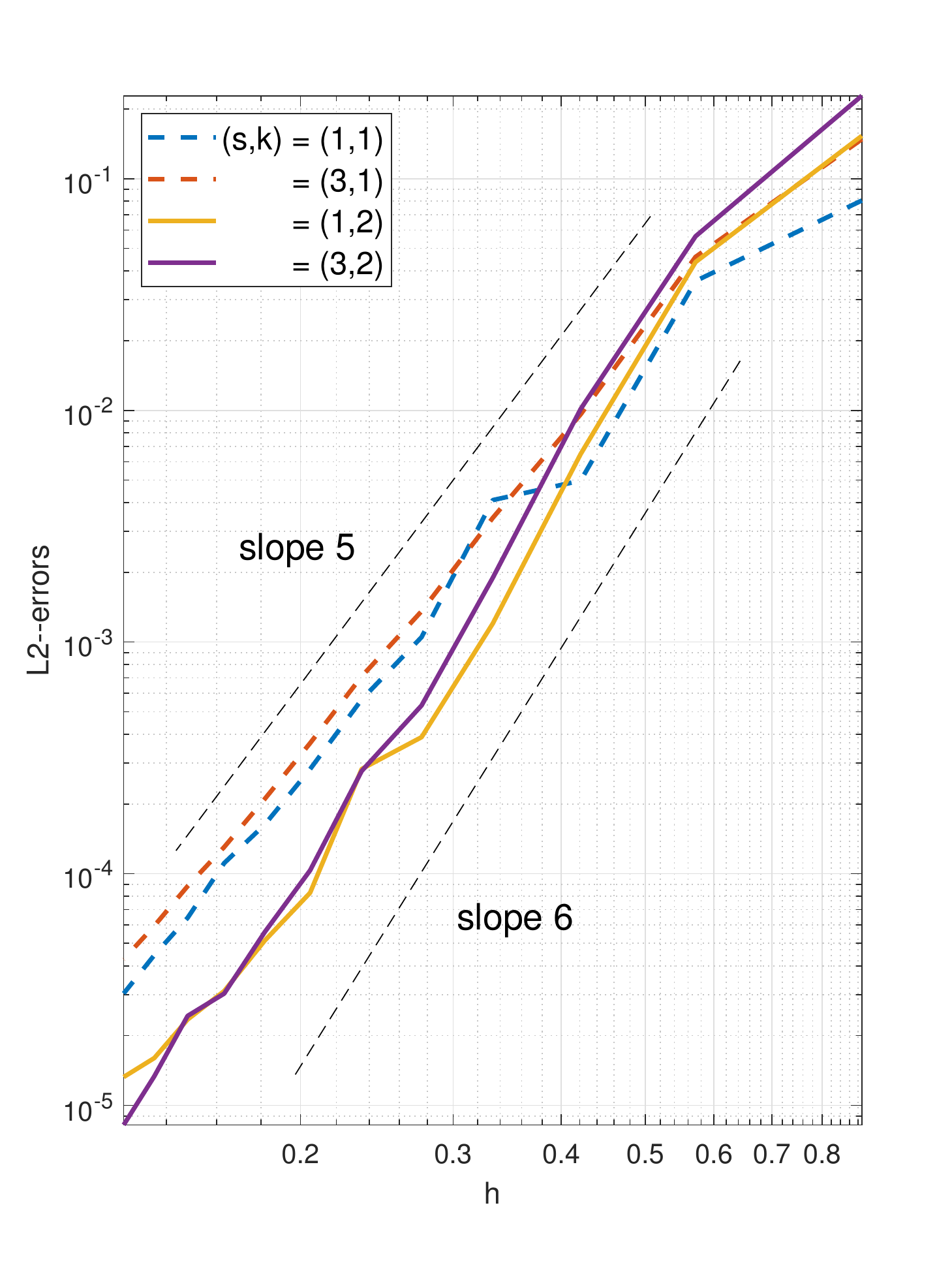}}\\
    \subfigure[$H^1$-error,
            $\mu=1$]{\includegraphics[width=.32\textwidth]{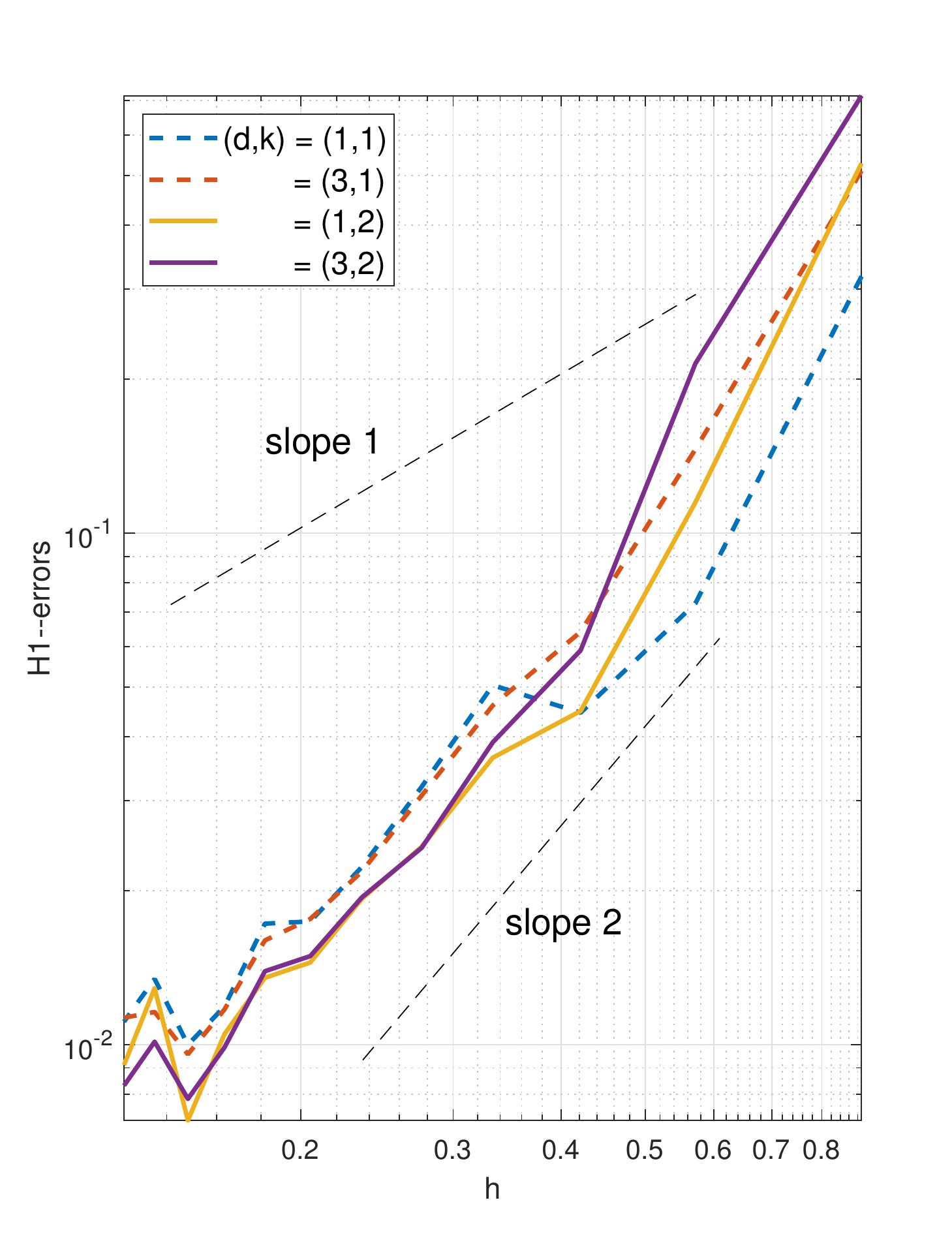}}
    \subfigure[$H^1$-error,
            $\mu=2$]{\includegraphics[width=.32\textwidth]{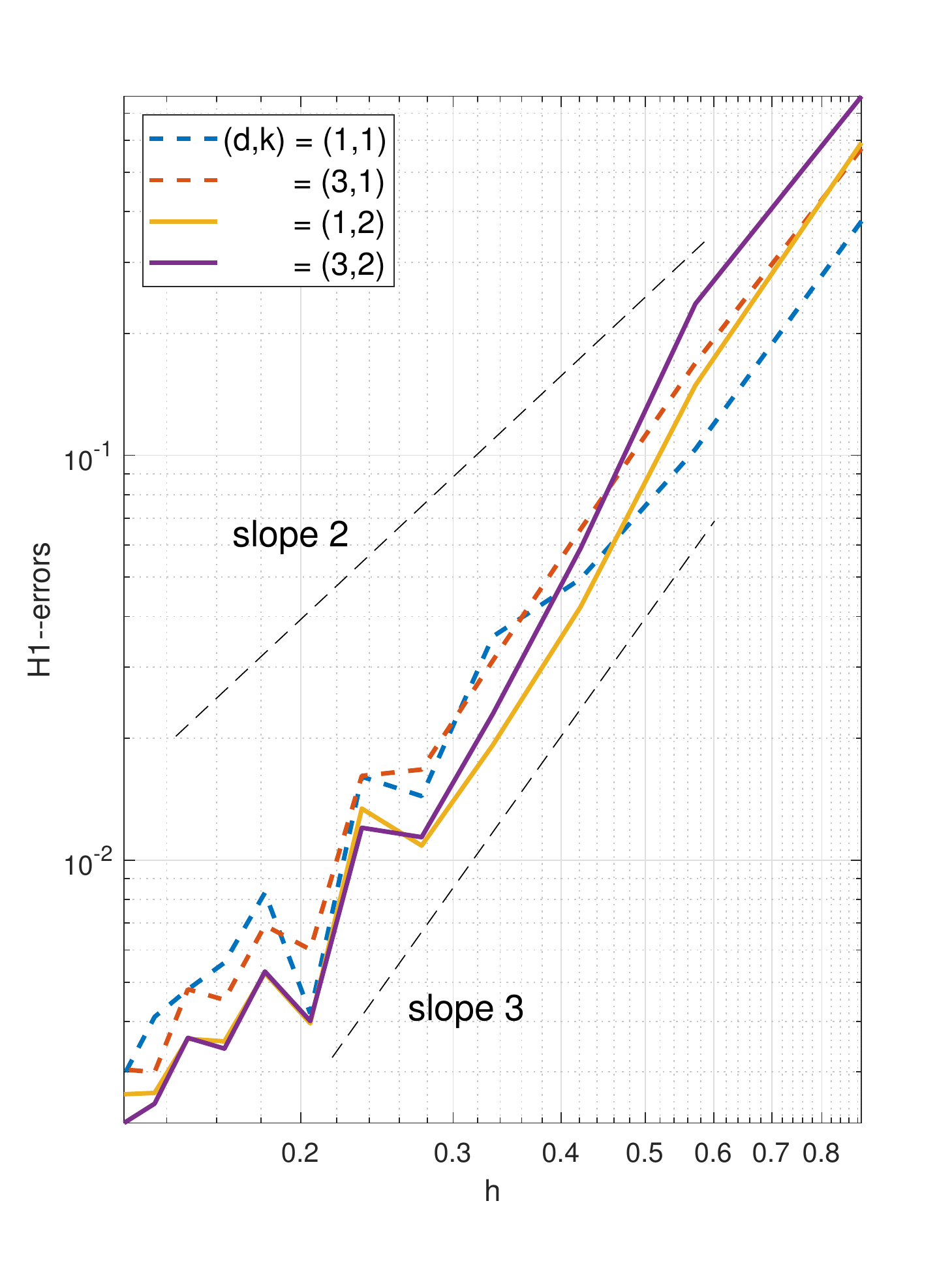}}
    \subfigure[$H^1$-error,
            $\mu=4$]{\includegraphics[width=.32\textwidth]{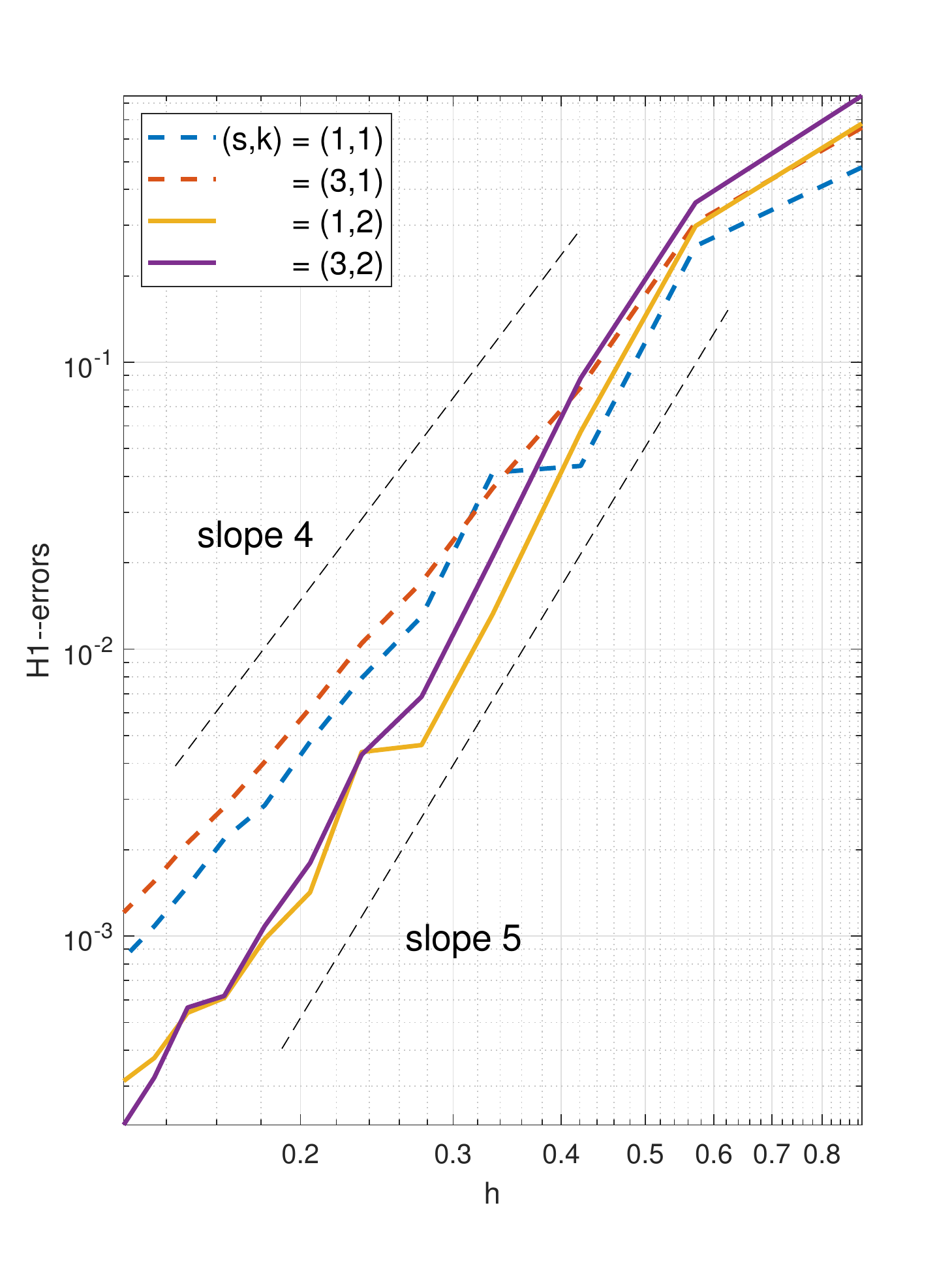}}
    \captionsetup{labelsep=period,labelfont=bf}
    \caption{\S\ref{exmp1}: Error profiles against fill distance $h_X$ for the fully discretized solution given by \eref{eq:NLWfulldiscrete2} subject to a $C^{\mu}(\Omega)$ initial conditions $u_0$ in \eref{eq:LW_initcond}.}
    \label{fig:LLconv}
\end{figure}

\begin{figure}[t]
    \centering
    \subfigure[$L^2$-error,
            $\mu=0$]{\includegraphics[width=.48\textwidth]{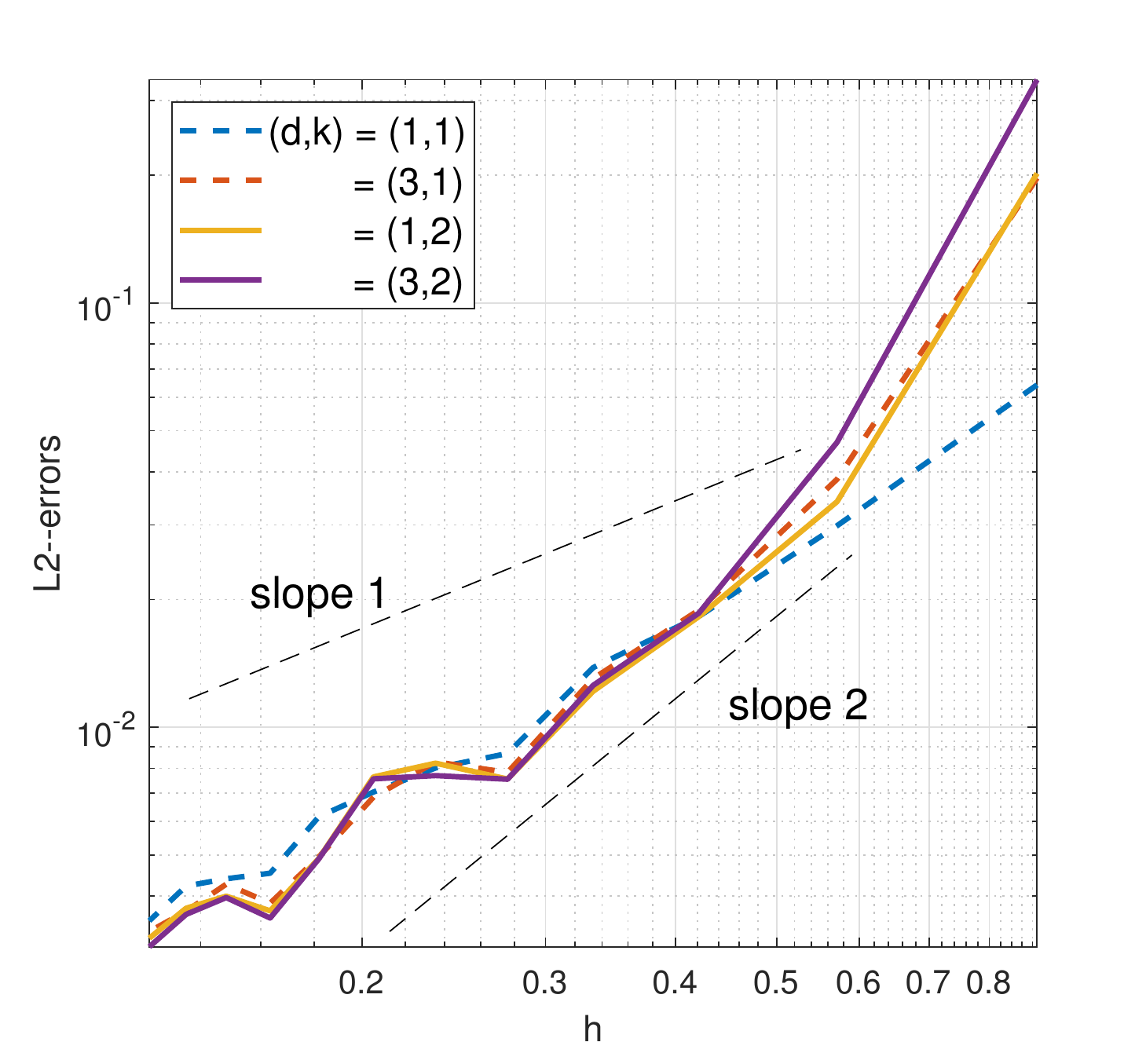}}
    \subfigure[Numerical solution]{\includegraphics[width=.48\textwidth]{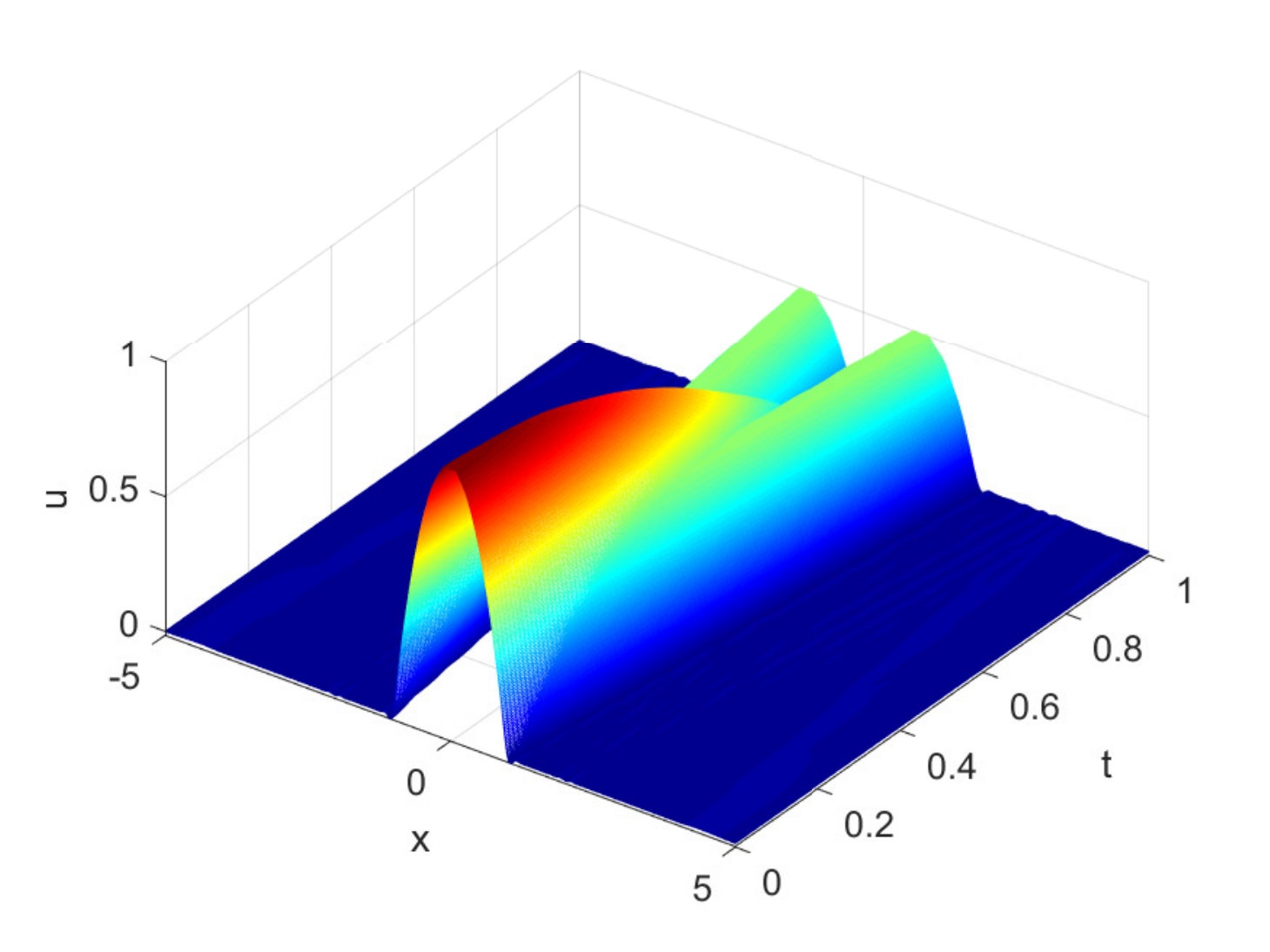}}\\
            \captionsetup{labelsep=period,labelfont=bf}
    \caption{\S\ref{exmp1}: (a) Error profiles against fill distance $h_X$ for the fully discretized solution given by \eref{eq:NLWfulldiscrete2} subject to $C^{0}(\Omega)$ initial conditions in \eref{eq:LW_initcond}; (b) Numerical solution at different time.}
    \label{fig:LLsolu}
\end{figure}

\fref{fig:LLconv}(c) shows the $L^2$-error for the $C^4$-smooth initial conditions  with $\mu=4$.
In this smoother test case, we clearly see that kernels with $k=2$ outperform those with $k=1$. In terms of accuracy, the value of $s$ is not important as expected based on what we know about restricted kernels. With two extra order of spatial smoothness in the initial condition (so is the exact solution), we see that $L^2$-convergence rate increase by two, comparing between  \fref{fig:LLconv}(b) and (c).
 The corresponding $\H^1$-error profiles are presented in \fref{fig:LLconv} (d) to (f). These  $\H^1$-error  profiles are   one order below that of their $L^2$ counterparts, which are consistent with our theoretical expectations.

To verify the robustness of the proposed method with respect to the smoothness of initial data, we also test the case of $\mu=0$.  We present the numerical results in  \fref{fig:LLsolu}; although Thm.~\ref{thm:NLWSemidiscrete} does not apply, the proposed method remains  $L^2$-convergent in this example with  a $C^0$ initial function.
Finally, we fix $n_X=100$ and kernel $\phi_{3,2}(r)$ to generate the temporal convergence results in \tref{Tab:TimeAccuracy}. This quick verification confirms that AVF method has the expected second order temporal convergence.

\begin{table}
    \centering
    \caption{\S\ref{exmp1}: Accuracy of the proposed method with kernel $\phi_{3,2}(r)$ verse time stepping size  for the linear wave equation
 with a $C^4(\Omega)$  initial condition.}
    \label{Tab:TimeAccuracy}
    \begin{tabular}{*{5}{c}}
        \toprule
        \qquad$\tau$ & \qquad $L^2$-error & \qquad rate &\qquad  $L^{\infty}$-error & \qquad rate \\
        \midrule
        \qquad 0.04  &\qquad 1.0689e-3& &\qquad 2.8914e-3 & \qquad \\
        \qquad 0.02& \qquad 2.6896e-4& \qquad 1.991 &\qquad 7.3049e-4 & \qquad 1.985 \\
        \qquad 0.01 &\qquad 6.7370e-5& \qquad 1.997  &\qquad 1.8701e-4 &\qquad 1.966  \\
        \qquad 0.005  & \qquad 1.7015e-5& \qquad 1.985 & \qquad 5.0886e-5 & \qquad 1.878 \\
        \bottomrule
    \end{tabular}
\end{table}

\begin{figure}
    \centering
    \includegraphics[width=.48\textwidth]{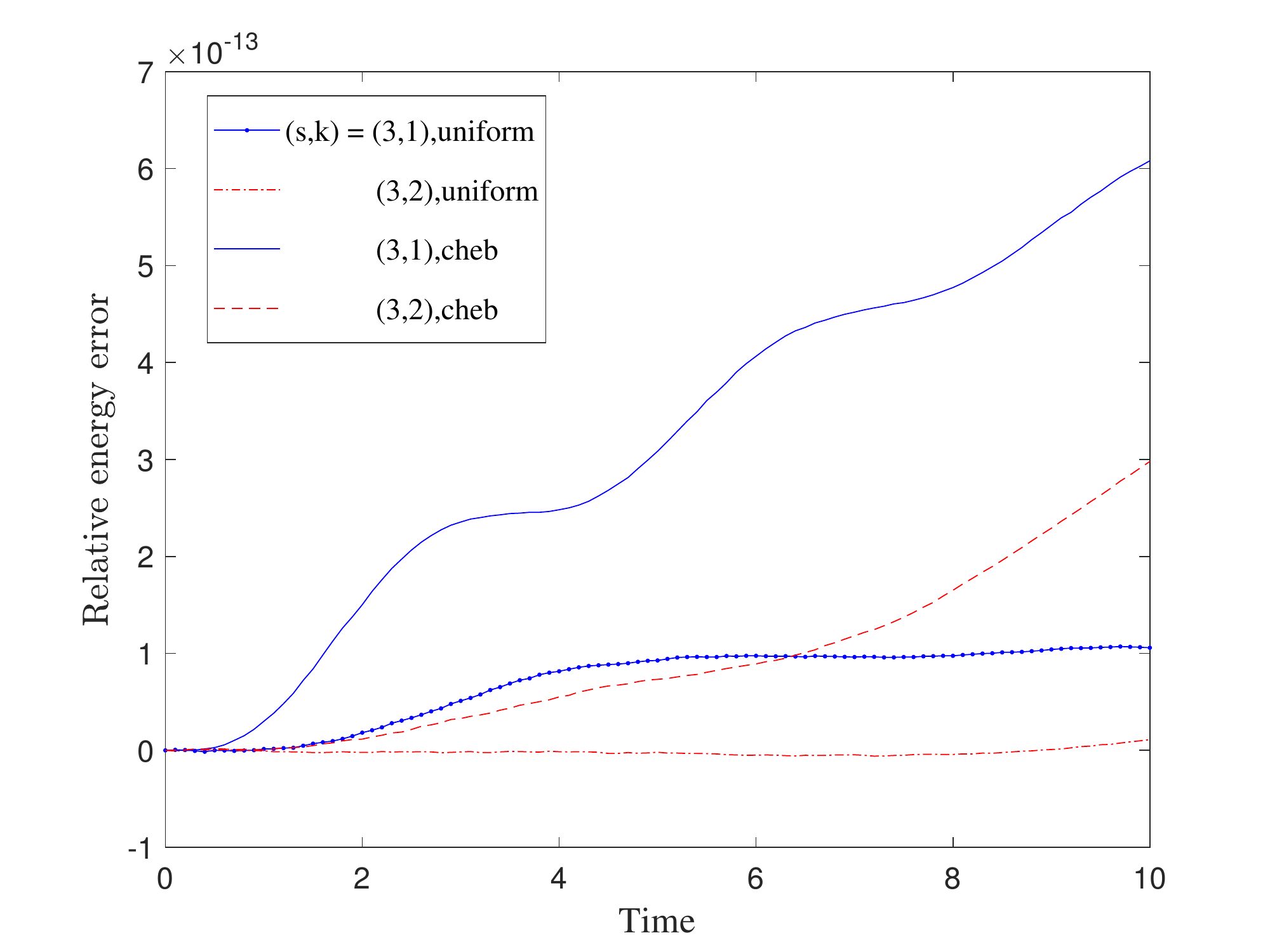}
    \captionsetup{labelsep=period,labelfont=bf}
    \caption{\S\ref{exmp1}: Discrete energy error of the proposed methods with uniform and Chebshev RBF centers for solving a linear wave equation over the time interval $[0,10]$.}
    \label{fig:LW_Energy}
\end{figure}

To verify the property of energy conservation, we solve the linear wave problem till a larger time $T=10$
(and  a larger $\Omega=[-12,12]$)
with
$n_X=100$ uniform and Chebshev RBF centers in $[-11,11]$.
Without explicitly using $W$ in the definition \eref{eq:NLWEnergy},
the \emph{relative discrete energy error} $E^d$
at time $0\leq t_n \leq T=10$ is calculated through the following formula
$$\emph{Discrete energy error}:=\frac{|E^d(t^n;u_X)-E^d(0;u_X)|}{|E^d(0;u_X)|}$$ also by adaptive integrations.
We take Chebshev centers to illustrate the effect of energy conservation on nonuniform centers.
\fref{fig:LW_Energy} presents the relative discrete energy errors at different times.
From  \fref{fig:LW_Energy},  we observe that  the  discrete energy error for Chebshev centers is slightly bigger than that of uniform centers, but they are both small enough and can be considered negligible. It shows that our meshless method can preserve the energy whether on uniform or nonuniform centers.

\subsection{Sine-Gordon equation and energy conservation}\label{exmp2}
To show the convergence results of our method for nonlinear equations, we consider the Sine-Gordon equation, namely $F'(u)=\sin(u)$, with the following exact solution
\begin{equation}\label{eq:NumerExamSGExact}
u^*(x,t)=4\tan^{-1}\left(\frac{1}{\zeta}\sinh\Big(\frac{\zeta t}{\sqrt{1-\zeta^2}}\Big)\Big/\cosh\Big(\frac{x}{\sqrt{1-\zeta^2}}\Big)\right).
\end{equation}
This is a kink-antikink system, an interaction between two solitons that move towards opposite directions with speed $\zeta\in(0,1)$. The solution becomes steeper as $\zeta\rightarrow 1$. In our experiment, we choose the speed $\zeta=0.9$.  Initial conditions were computed using \eref{eq:NumerExamSGExact}.
We follow the numerical set up as in \sref{exmp1}, but only focus on kernel $\phi_{s,k}$ with $s=3$ and $k\in\{1,2\}$. On top of that, we use a fixed-point iteration method to solve the nonlinear ODE \eref{eq:NLWfulldiscrete2}.
\tref{Tab:SG} verifies that our observations in linear wave  (see \fref{fig:LLconv}(c) and (f) for results with sufficiently smooth solution) also hold in Sine-Gordon equation up to $T=1$ on domain $[-20,20]$.

%
%
\begin{table}
    \centering
    \caption{\S\ref{exmp2}: Numerical error of the proposal method for solving a Sine-Gordon equation.}
    \label{Tab:SG}
    \begin{tabular}{*{9}{c}}
        \toprule
        \multirow{2}*{N} & \multicolumn{4}{c}{$\phi_{3,1}(r)$ 
} & \multicolumn{4}{c}{$\phi_{3,2}(r)$                        
}\\
        \cmidrule(lr){2-5}\cmidrule(lr){6-9}
        & $L^2$-error & rate & $H^1$-error & rate & $L^2$-error & rate & $H^1$-error & rate\\
        \midrule
        50 &2.0018e-2& &8.7515e-2 & &2.8107e-2&      &1.2419e-1& \\
        100&1.1385e-3&4.14 &9.1802e-3 &3.25 &4.4244e-4&5.99 &3.7351e-3&5.06\\
        150&2.2130e-4&4.04 &2.5820e-3 &3.13 &3.9116e-5&5.98 &4.7679e-4&5.08 \\
        200&7.2005e-4&3.90 &1.0957e-3 &2.98 &6.9956e-6&5.98 &1.1117e-4&5.06\\
        \bottomrule
    \end{tabular}
\end{table}

We now compare the energy-conservation
performance of the proposed meshless Galerkin method together with the AVF time integrator (MGAVF) against three other  methods: the RBF interpolation method for spatial discretization together with the implicit midpoint method for temporal discretization (RBFIM), the traditional meshless Galerkin method with the AVF time integrator (TGAVF), and the method proposed in \cite{kunemund2019high}, which uses the traditional RBF Galerkin method with a modified Crank-Nicolson scheme (TGCN). Here, the traditional meshless Galerkin method directly employs the Galerkin equation \eqref{eq:NLWGalerkin} without the projection step \eqref{eq:NLWGalerkinDiscrete2}.
We use kernel $\phi_{3,2}(r)$ with time step $\tau =0.01$, and $n_X=200$ RBF centers in all methods.
\fref{fig:SGLongTime}(a) shows that all but RBFIM are having similar performance in terms of $L^2$-error.  In  \fref{fig:SGLongTime}(b), it is evidential that the proposed method (MGAVF) is the only one that can  preserve the discrete energy up to high precision.


\begin{figure}
    \centering
    \subfigure[Absolute $L^2$ error]{\includegraphics[width=.49\textwidth]{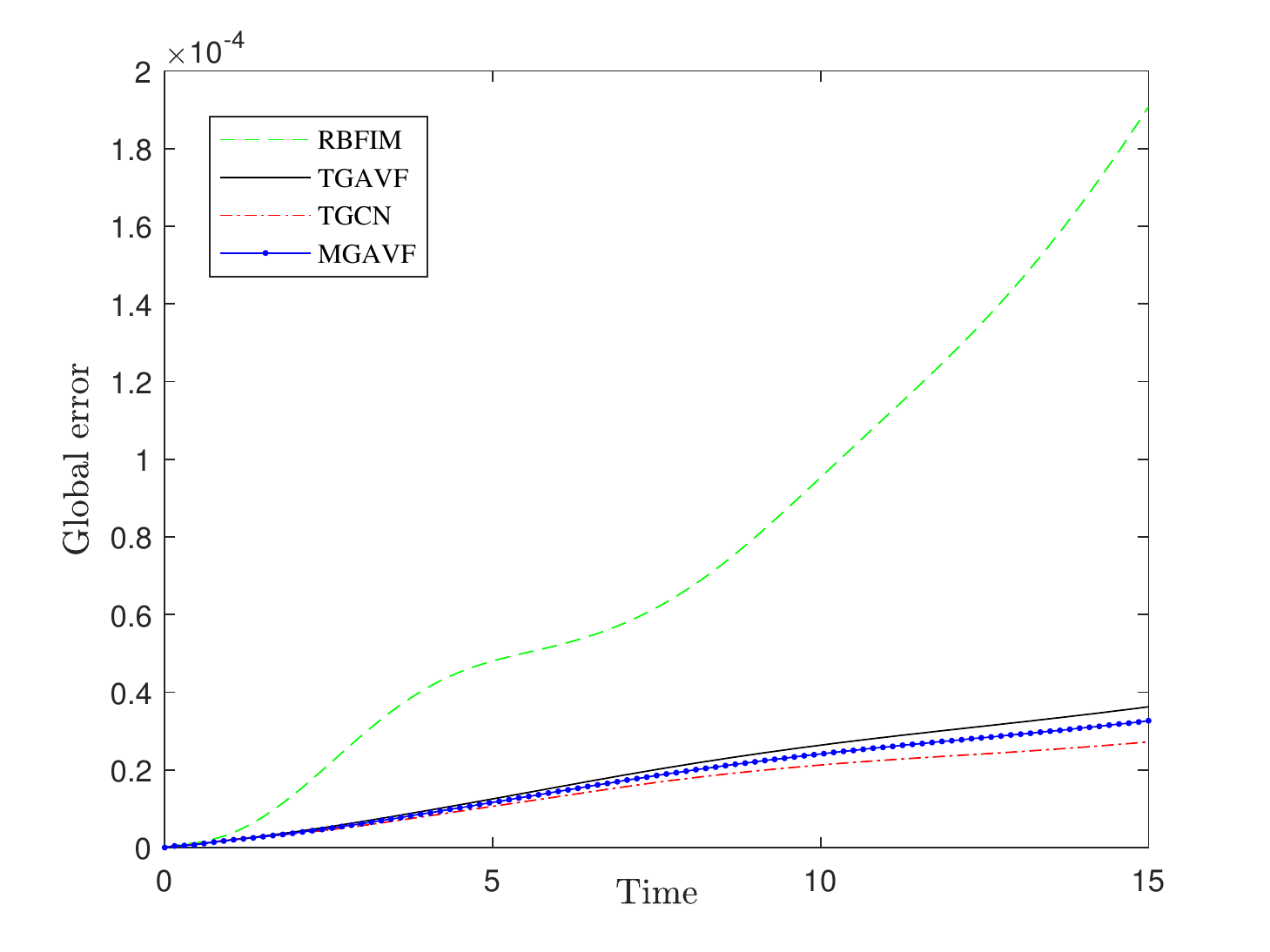}}
    \subfigure[Relative energy error]{\includegraphics[width=.49\textwidth]{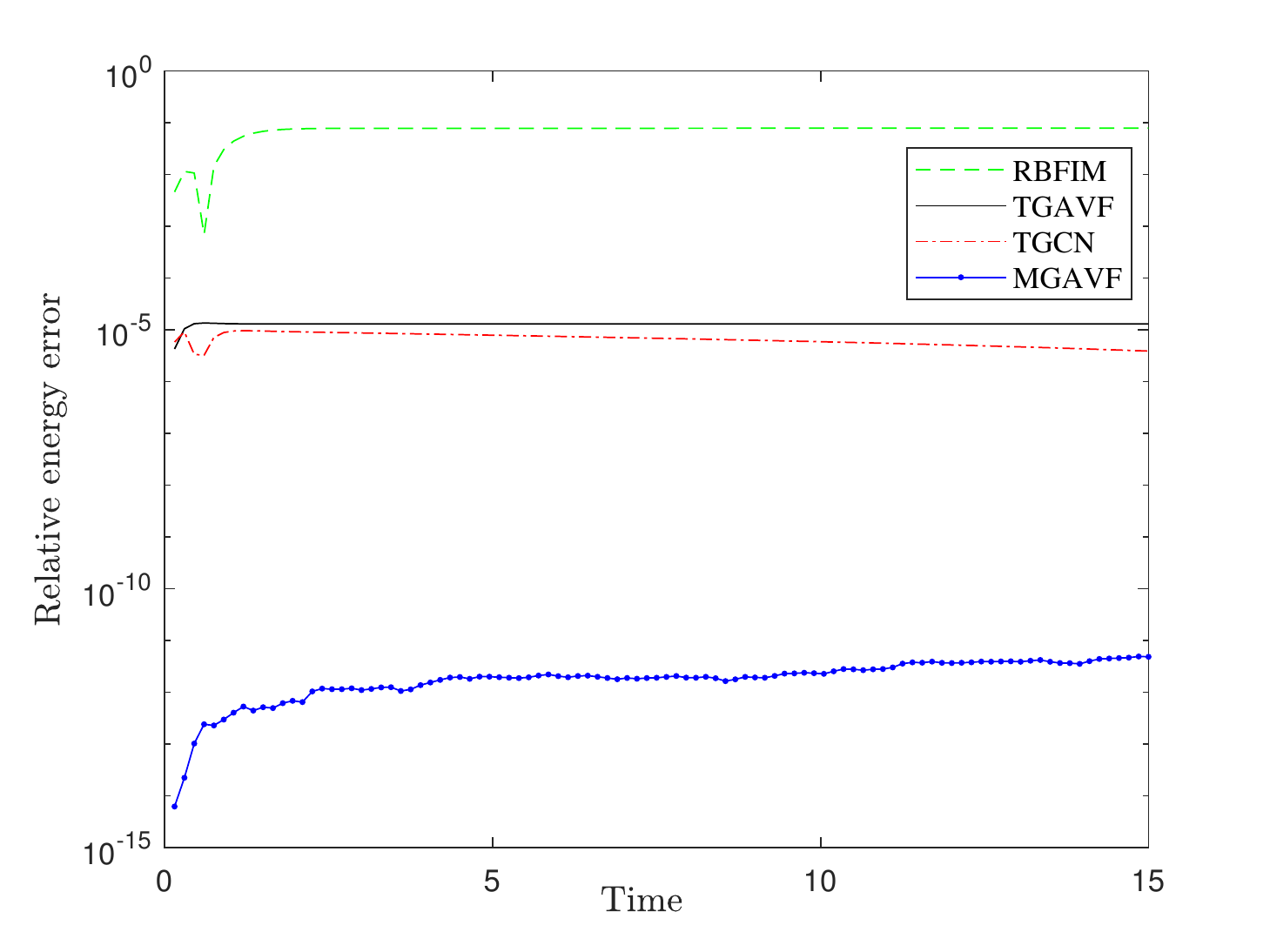}}
    \captionsetup{labelsep=period,labelfont=bf}
    \caption{\S\ref{exmp2}: Error and Energy conservation of the proposed method (MGAVF) against  RBFIM, MGIM, and MGCN.}
    \label{fig:SGLongTime}
\end{figure}

\subsection{Simulating 2D Klein-Gordon equation}\label{exmp3}
The important feature of meshless methods is their flexibility to process multi-dimensional problems. Thus, we desire to solve another two-dimensional nonlinear wave equation to conclude this paper.

We focus on the nonlinear function $F'(u)=u^3$ in \eqref{eq:NLW} and consider
\begin{equation}\label{eq:2DKleinGordon}
 \dtt{u}-u_{xx}-u_{yy}+u^3=0,
\end{equation}
which is the well-known Klein-Gordon equation, subject to initial conditions
$$u(x,0)=2\text{sech}(\cosh(x^2+y^2)),\quad \dt{u}(x,0)=0.$$
We solve this PDE on both $n_X=41^2$ uniformly spaced and scattered Halton RBF centers  \cite[Appx. A]{fasshauer2007meshfree} in $\Sigma=[-10,10]^2$
with  $\phi_{s,k}$ ($s=3,k=2$).  We simply take $n_Z=128^2$ regular grids in $\Omega=[-11,11]^2$  as the set of quadrature points $Z$.


\fref{fig:NLW2d_Solution} shows a few snapshots of the numerical solution at different times. We see  the correct physics that a circular soliton expands and propagates to the whole domain. These results are consistent with those in the literature \cite{darani2017direct,cai2020two}.
More importantly,   \fref{fig:NLW2D_Energy} shows the relative energy of our numerical solution over time. Energy preservation on both uniform and scattered RBF centers are of the order of $1E-11$.

\begin{figure}[tp]
    \centering
    \subfigure[$T=0$]{\includegraphics[width=.49\textwidth]{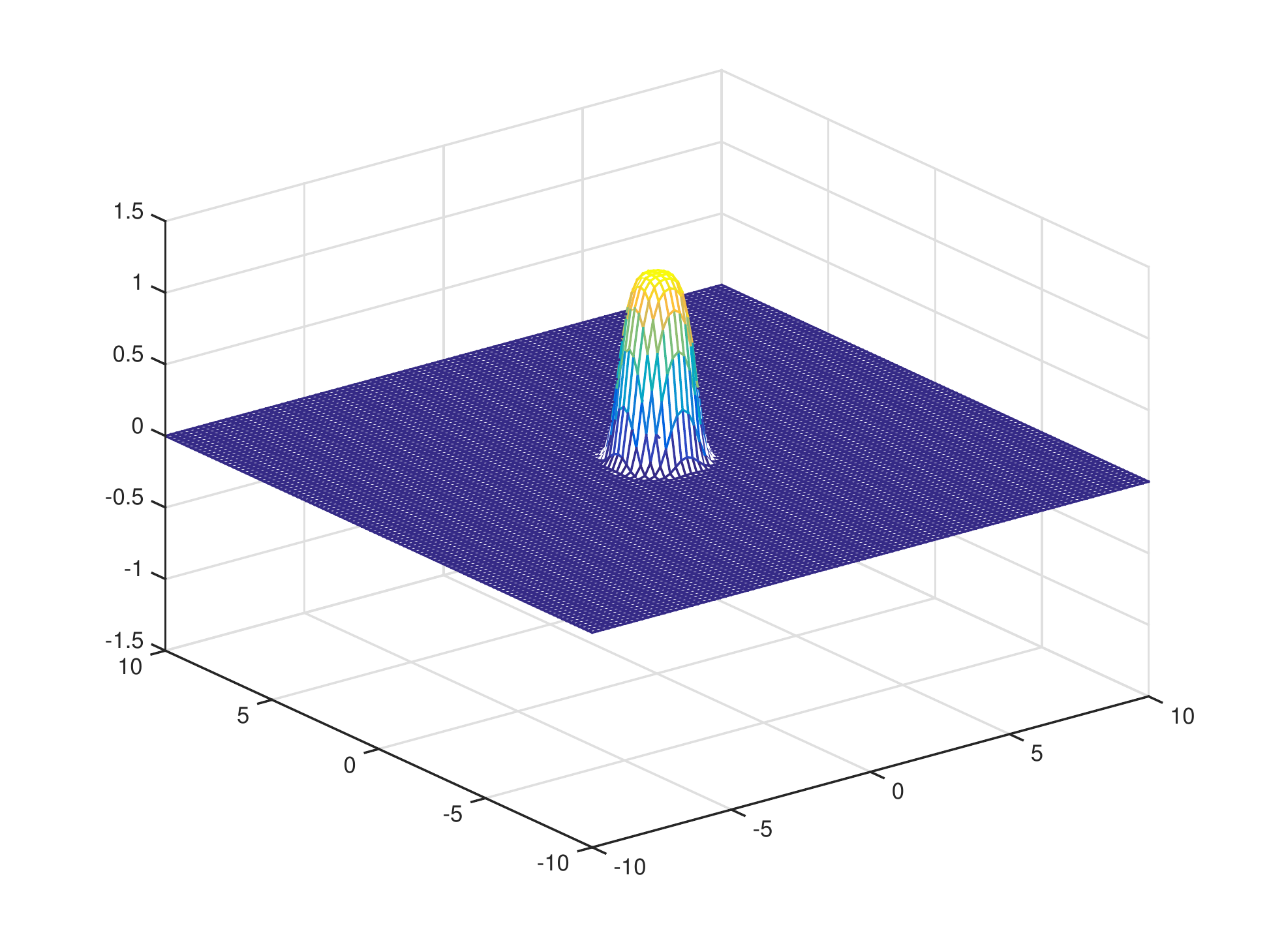}}
    \subfigure[$T=1$]{\includegraphics[width=.49\textwidth]{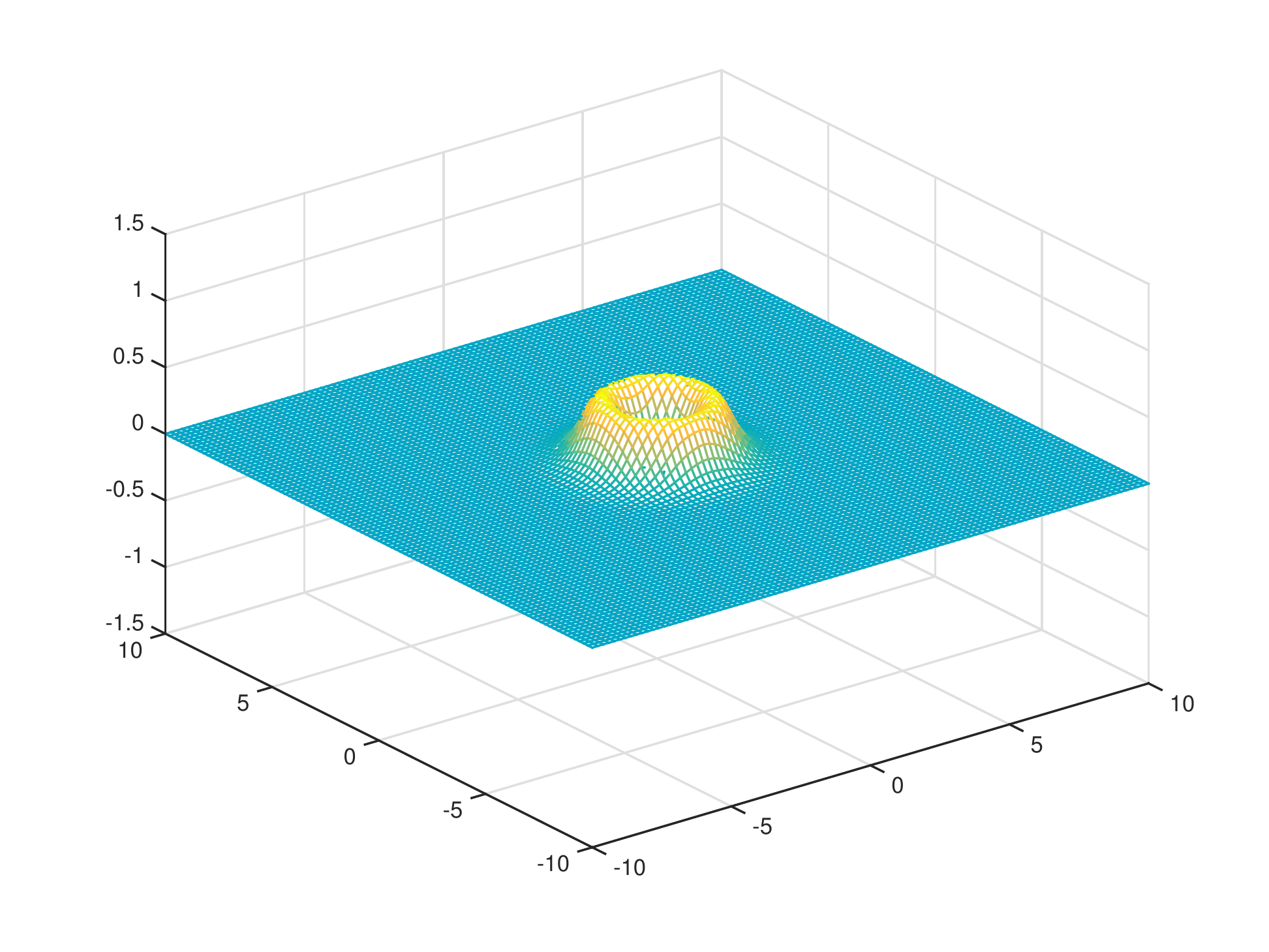}}
    \subfigure[$T=3$]{\includegraphics[width=.49\textwidth]{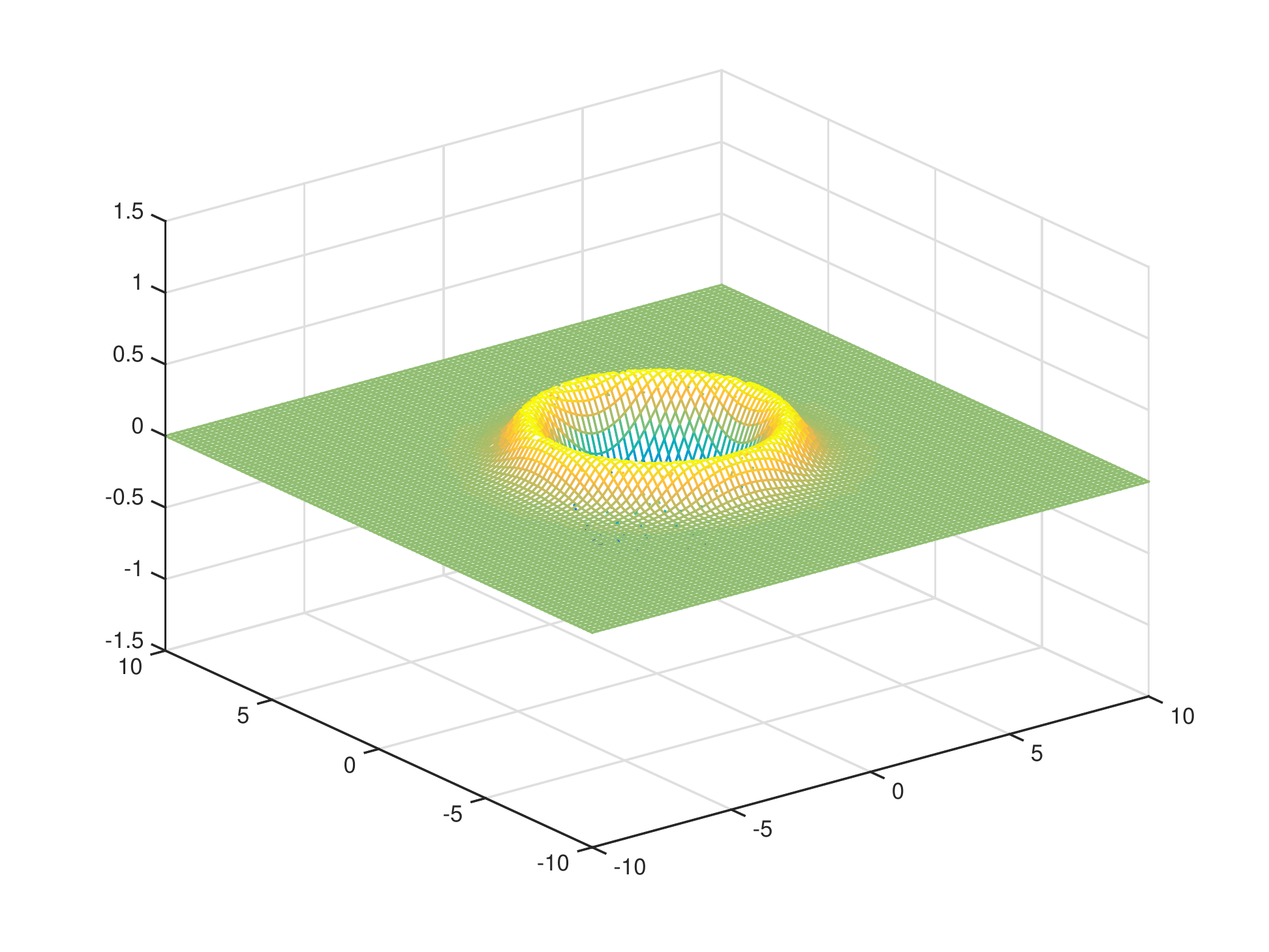}}
    \subfigure[$T=7$]{\includegraphics[width=.49\textwidth]{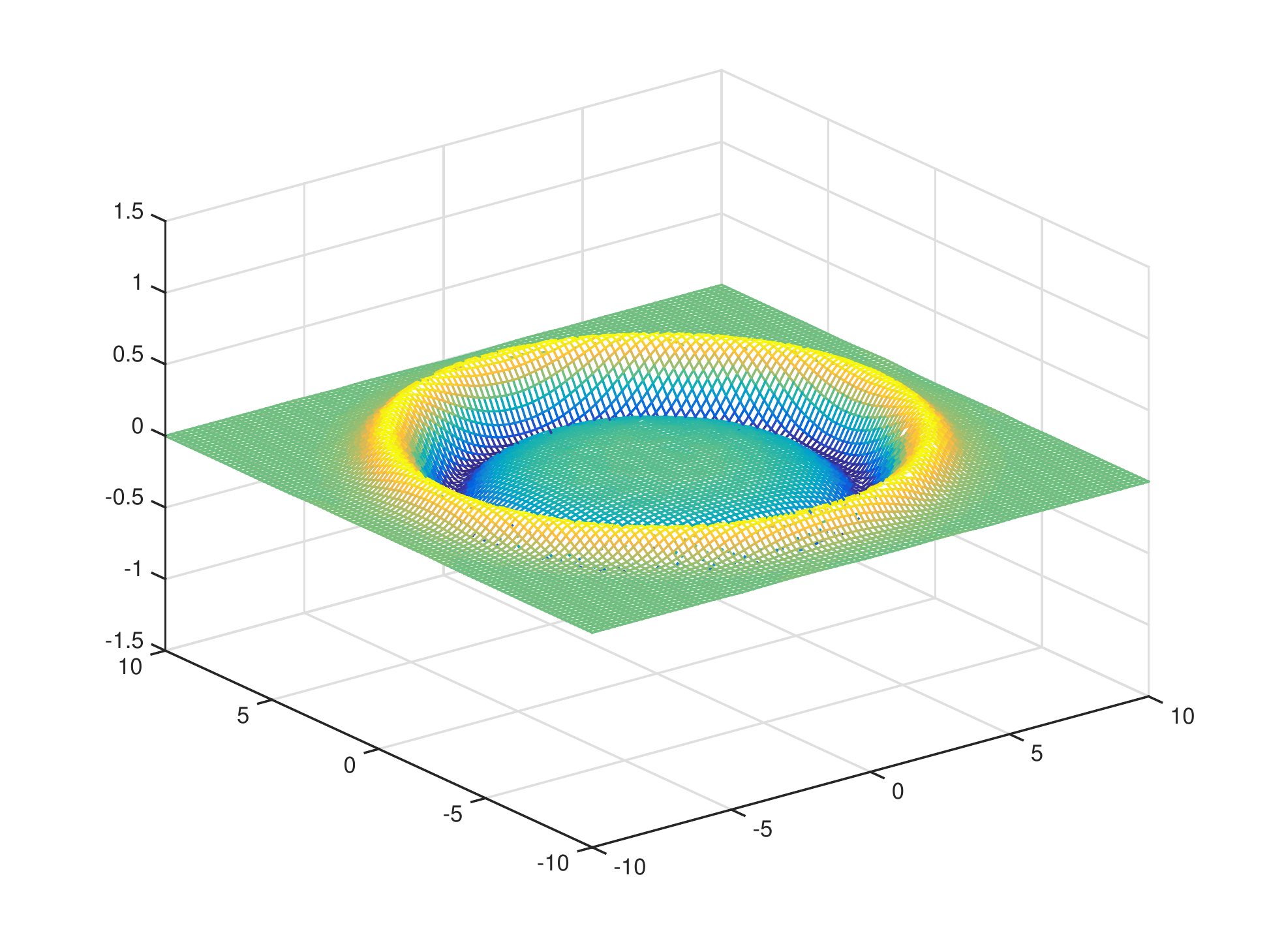}}
    \captionsetup{labelsep=period,labelfont=bf}
    \caption{\S\ref{exmp3}: 2D Klein-Gordon equation: Snapshots of the solution at different times with $n_X=41^2$ Halton nodes and $\tau=0.001$. Numerical solution interpolated on a equidistant grid of 101 nodes in each space direction.}
    \label{fig:NLW2d_Solution}
\end{figure}

\begin{figure}[htp]
    \centering
    \includegraphics[width=.48\textwidth]{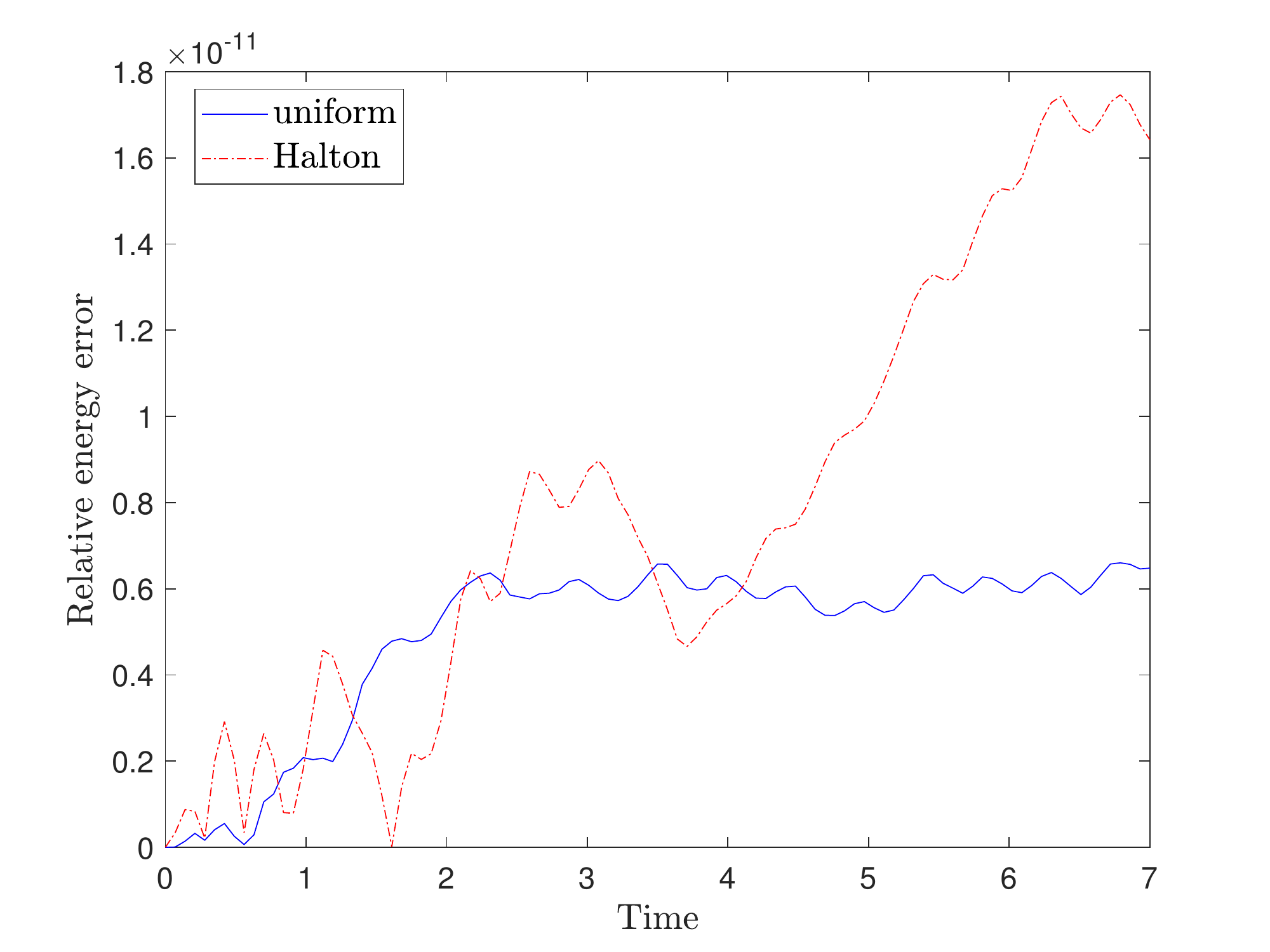}
    \captionsetup{labelsep=period,labelfont=bf}
    \caption{\S\ref{exmp3}:
Numerical errors in the discrete energy over the time interval $[0,7]$ when the proposed method is employed to simulate a 2D Klein-Gordon equation.}
    \label{fig:NLW2D_Energy}
\end{figure}

\section{Conclusion}
We propose a meshless conservative Galerkin method to solve semi-linear wave equations, where the RBF Galerkin method in space and AVF method in time discretization is implemented.
Our formulation only differs slightly to the traditional meshless Galerkin method, but conserves energy by introducing some appropriate projections.
We further show that a discrete energy,  whose accuracy can independently be controlled by selecting quadrature points and weights, is conserved exactly by the semi-discretized solution.
Convergence results for solving Hamiltonian wave equations are also investigated.
Various numerical simulations for solving different types of PDEs are presented to verify the accuracy and energy conservation properties.
Compared with classical conservative schemes, our method can conserve the energy on scattered nodes in multi-dimensional space.
This study is only focused on time-independent scattered node distributions.
However, it is well-known that RBF methods are very competitive in an adaptive setting. Thus, adaptive energy conserving methods based on RBF approximation will be carried out in the future work.


\bibliographystyle{plain}
\bibliography{meshlessConservative}

\end{document}